\documentclass[reqno]{amsart}
\usepackage{hyperref}
\usepackage{amssymb}

\begin{document}
\title[Nonlocal Gross-Pitaevskii equation]
{Global solutions for 3D nonlocal Gross-Pitaevskii equations with rough data}

\author[H. Pecher]
{Hartmut Pecher}

\address{Hartmut Pecher \newline
Fachbereich Mathematik und Naturwissenschaften\\
Bergische Universit\"at Wuppertal\\
Gau{\ss}str.  20, 42097 Wuppertal, Germany}
\email{pecher@math.uni-wuppertal.de}

\keywords{Gross-Pitaevskii equation; global well-posedness;
\hfill\break\indent  Fourier restriction norm method}

\begin{abstract}
 We study the Cauchy problem for the Gross-Pitaevskii equation with a nonlocal
 interaction potential of Hartree type in three space dimensions.
 If the potential is even and positive definite or a positive function and
 its Fourier transform decays sufficiently rapidly the problem is shown to
 be globally well-posed for large rough data which not necessarily have
 finite energy and also in a situation where the energy functional is not
 positive definite. The proof uses a suitable modification of the I-method.
\end{abstract}

\maketitle
\numberwithin{equation}{section}
\newtheorem{theorem}{Theorem}[section]
\newtheorem{lemma}[theorem]{Lemma}
\newtheorem{proposition}[theorem]{Proposition}
\allowdisplaybreaks

\section{Introduction and main results}

We consider the Cauchy problem for the Gross-Pitaevskii
equation with nonlocal nonlinearity in three space dimensions:
\begin{gather}\label{1}
i\frac{\partial v}{\partial t} - \Delta v  =  v(W * (1-|v|^2))\\
\label{2} v(x,0)  =  v_0(x) \,,
\end{gather}
under the condition
\begin{equation} \label{3}
 v \to 1  \quad \text{as }  |x| \to + \infty  \,,
\end{equation}
where $v : \mathbb{R}^{1+3} \to {\mathbb C} $.

More generally one could also consider the condition
\begin{equation}
\label{3'}
 |v| \to 1  \quad \text{as }  |x| \to + \infty  \,,
\end{equation}
but for simplicity we restrict ourselves to \eqref{3}.
This problem was introduced by Gross \cite{Gr} and Pitaevskii \cite{P} 
for modeling the kinetic of a weakly interacting Bose gas. Here $W$ describes 
the interaction between bosons. The original equation reads as follows
$$ 
i \frac{\partial \psi}{\partial t}(x,t) + \frac{\hbar ^2}{2m} \Delta \psi(x,t) 
= \psi(x,t) \int_{\mathbb{R}^n} V(x-y) |\psi(y,t)|^2 dy
$$
and is equivalent modulo normalizations to equation \eqref{1}, provided $W*1$ 
is well-defined and positive, which in the cases we consider 
(under the assumptions (A1) and either (A2) or (A3) below) is obviously true. 
For a detailed derivation of our problem from the original Gross-Pitaevskii 
form we refer to \cite{L}.

The most studied case is $W= \delta$ (= Dirac function), which occurs in 
theoretical physics e.g. in superfluidity, nonlinear optics and Bose-Einstein 
condensation \cite{C},\cite{JPRo},\cite{JR},\cite{KL}. 
For further references we also refer to the introduction of \cite{L}.

Using the energy conservation law which in the case $W=\delta$ is
$$
E(v(t)) = \int (|\nabla v(x,t)|^2 + \frac{1}{2} (|v(x,t)|^2 -1)^2) dx = E(v_0)\,,
$$
it was shown by Bethuel and Saut \cite{BS}, Appendix A, that the problem is 
globally well-posed for data of the form $v_0 \in 1+H^1(\mathbb{R}^3)$.
 G\'erard \cite{Ge} proved the same result for data in the larger energy 
space in two and three space dimensions. Gallo \cite{Ga} generalized these 
results to a class of local nonlinearities for data with finite energy and 
space dimension $n \le 4$. Global well-posedness for the Gross-Pitaevskii 
equation in the case $n=4$ in the (critical) energy space was proven by
 Killip, Oh, Pocopnicu and Visan \cite{KOPV}, a case which was not considered 
in Gallo's paper. The author \cite{Pe} showed that global well-posedness holds
 true even for data with less regularity, namely $v_0 = 1+u_0$, where 
$u_0 \in H^s(\mathbb{R}^3)$ for $5/6 < s < 1$.
To prove this result one uses Bourgain type spaces and the so-called $I$-method 
(or method of almost conservation laws), which was introduced by 
Colliander, Keel, Staffilani, Takaoka and Tao \cite{CKSTT} and successfully
applied to various problems.

We now want to study the problem for two types of nonlocal nonlinearities. 
Nonlocal nonlinearities were as mentioned above already introduced by Gross 
and Pitaevskii. In the case of three space dimensions Shchesnovich and 
Kraenkel \cite{SK} consider 
$W(x) = \frac{1}{4 \pi \epsilon^2 |x|} \exp(-\frac{|x|}{\epsilon})$ for
 $\epsilon > 0$ with Fourier transform 
$\widehat{W}(\xi) = \frac{1}{1+\epsilon^2  |\xi|^2}$. The case 
$ W = \chi_{\{|x| \le a\}}$ ($\chi_A$ = characteristic function of the 
set $A$) was used in the study of supersolids \cite{ABJ,JPR,PR}. 
These examples are included in the class of nonlocal nonlinearities with 
suitable mapping properties and positivity conditions on $W$ considered 
by de Laire \cite{L} such that the Cauchy problem \eqref{1},\eqref{2},\eqref{3'} 
is globally well-posed in the space $\phi + H^1(\mathbb{R}^n)$, where
$\phi$ has finite energy and fulfills suitable boundedness assumptions, 
in particular $|\phi(x)| \to 1 $ as $|x| \to \infty$.

Our aim is to give similar results for less regular data. From now on we consider
 the case of three space dimensions and make the following assumptinos:

\textbf{General Assumption on $W$:}
\begin{itemize}
\item[(A1)] $W:  \mathbb{R}^3 \to \mathbb{R}$ is even,
$W \in L^1(\mathbb{R}^3)$, $|\widehat{W}(\xi)| \lesssim \langle \xi \rangle^{-2}$
for all $\xi \in \mathbb{R}^3$
and \textbf{either} 
\item[ (A2)] $\widehat{W}(\xi) \ge 0$ for all $\xi \in \mathbb{R}^3$ 
\textbf{or}
\item[(A3)] $W(x) \ge 0$ \for all $x \in \mathbb{R}^3$.
\end{itemize}
Let us remark, that $\widehat{W}$ is real-valued and even, if $W$ has the 
same properties.

We have especially the following two examples in mind, which we mentioned above: 

\noindent\textbf{Case A:} $W(x) = \frac{1}{4 \pi |x|} e^{-|x|}$.
We have
$\widehat{W}(\xi) = \frac{1}{1+|\xi|^2}$, so that (A1),(A2) and (A3) are satisfied.

\noindent \textbf{Case B:} $ W = \chi_{\{|x| \le a\}}$.
Obviously (A3) is satisfied. We also have
$\widehat{W}(\xi) = a^{-\frac{3}{2}} |\xi|^{-\frac{3}{2}} J_{\frac{3}{2}}(2\pi a|\xi|)$,
 where $J_{\frac{3}{2}}$ is the Bessel function of the first kind of order
 $\frac{3}{2}$, which has the properties 
$J_{\frac{3}{2}}(|\xi|) \sim |\xi|^{3/2}$ as $|\xi| \ll 1$ and
$J_{\frac{3}{2}}(|\xi|) \lesssim \frac{1}{|\xi|^{1/2}}$ as $|\xi| \gg 1$.
Thus (A1) is also satisfied.

For simplicity we assume $\phi \equiv 1$ and consider \eqref{1},\eqref{2},\eqref{3}.
As usual we transform the problem \eqref{1},\eqref{2},\eqref{3} by setting $u=v-1$ 
into the equivalent form
\begin{gather}\label{4}
i \frac{\partial u}{\partial t} - \Delta u + F(u)  =  0 \\
\label{5} u(x,0)  = u_0(x)
\end{gather}
where
\begin{equation}\label{6}
F(u) = (1+u)(W*(|u|^2 +2\, Re\, u))
\end{equation}
under the condition
\begin{equation} \label{7}
u \to 0 \quad \text{as } |x| \to + \infty \,.
\end{equation}
Assuming $W$ to be real-valued and even the conserved energy is given by
\begin{equation}\label{E'}
E(u(t)) = \int |\nabla u(t)|^2 dx + \frac{1}{2} \int (W*(|u|^2 
+ 2 \, \operatorname{Re} u))(|u|^2 + 2 \, \operatorname{Re} u) dx \,.
\end{equation}
We remark that no $L^2$-conservation law holds.

Under our hypothesis on $W$ de Laire's results \cite{L} especially imply that 
the Cauchy problem \eqref{4},\eqref{5},\eqref{6},\eqref{7} is globally 
well-posed in $C^0(\mathbb{R},H^1(\mathbb{R}^3))$ for data
$u_0 \in H^1(\mathbb{R}^3)$. We now show that this problem for data
 $u_0 \in H^s(\mathbb{R}^3)$ is globally well-posed in
$C^0(\mathbb{R},H^s(\mathbb{R}^3))$, i.e. \eqref{1},\eqref{2},\eqref{3} for
$v_0 \in 1+H^s(\mathbb{R}^3)$, if $ 1/2 < s < 1 $ by application of the $I$-method.
As usual the energy conservation law is not directly applicable for $H^s$-data with 
$s<1$. However
there is an almost conservation law for the modified energy $E(Iu)$, which
is well defined for $u \in H^s$ (see the definition of $I$ below). 
If we assume (A1) and (A2), this leads to an a-priori bound of 
$ \|\nabla Iu(t)\|_{L^2}$, if $s$ is close enough to 1,
namely $s > 1/2$, because the energy functional is positive definite, a 
property which is usually assumed when the $I$-method is applied.  
If we assume (A1) and (A3) however it is not obvious that the $H^1$-norm of the 
solution can be controlled by the energy, because it is not definite. 
Nevertheless it is possible to modify the $I$-method in this case suitably, 
but the argument to get the required bound for
$ \|\nabla Iu(t)\|_{L^2}$ is more involved. Once a bound for 
$ \|\nabla Iu(t)\|_{L^2}$ is achieved we can also deduce an a-priori bound for 
$\|u(t)\|_{L^2}$, which together gives an a-priori bound for $\|u(t)\|_{H^s}$.

The main results (cf. the definition of the $X^{s,b}$-spaces below) are summarized 
in the following three theorems:

\begin{theorem}[Unconditional uniqueness] \label{Theorem 1.1}
Assume {\rm (A1)} and moreover {\rm (A2)} or {\rm (A3)}, 
$u_0 \in L^2(\mathbb{R}^3)$. The Cauchy problem \eqref{4},\eqref{5},\eqref{6}
has at most one solution
$u \in C^0([0,T],L^2(\mathbb{R}^3))$ for any $T>0$.
\end{theorem}

\begin{theorem}[Local well-posedness] \label{Theorem 1.2}
Assume {\rm (A1)} and moreover {\rm (A2)} or {\rm (A3)},
 $s \ge 0$ and $u_0 \in H^s(\mathbb{R}^3)$. Then the Cauchy problem
\eqref{4},\eqref{5},\eqref{6} has a unique local solution 
$u \in X^{s,\frac{1}{2}+}[0,\delta]$, where $\delta$ can be chosen as 
$ \delta \sim \|u_0\|_{H^s}^{-\frac{4}{2s+1}-}$. This solution belongs to 
$C^0([0,\delta],H^s(\mathbb{R}^3))$ and is also unique in this space.
\end{theorem}

\begin{theorem}[Global well-posedness] \label{Theorem 1.3}
Assume {\rm (A1)} and moreover {\rm (A2)} or {\rm (A3)}, $ T >0$, $s > 1/2$ 
and $u_0 \in H^s(\mathbb{R}^3)$. Then the Cauchy problem
 \eqref{4},\eqref{5},\eqref{6} has a unique global solution 
$u \in X^{s,\frac{1}{2}+}[0,T]$. This solution belongs to 
$C^0([0,T],H^s(\mathbb{R}^3))$ and is also unique in this space.
\end{theorem}

We use the following notation and well-known facts: the multiplier $I=I_N$ is
for given $s<1$ and $N \ge 1$ defined by
$$ 
\widehat{I_N f}(\xi) := m_N(\xi) \widehat{f}(\xi) \,,
$$
where $ \widehat{\enspace}  $ denotes the Fourier transform with respect to the
space
variables. Here $m_N(\xi)$ is a smooth, radially symmetric, nonincreasing
function of $|\xi|$ with
$$
 m_N(\xi) = \begin{cases}    
 1 & |\xi|\le N \\  
 (\frac{N}{|\xi|})^{1-s}  &    |\xi| \ge 2N .
\end{cases} 
$$
We remark that $I: H^s \to H^1$ is a smoothing operator, so that especially
$E(Iu)$ is well-defined for $u \in H^s(\mathbb{R}^3)$.
This follows from $W \in L^1$, Young's inequality and Sobolev's embedding
$H^1(\mathbb{R}^3) \subset L^4(\mathbb{R}^3)$.

We use the Bourgain type function space $X^{m,b}$ belonging to the
Schr\"odinger equation $iu_t -\Delta u = 0$, which is defined as follows:
 let $ \widehat{\enspace}$ or $\mathcal{F}$ denote
the Fourier transform with respect to space and time and $\mathcal {F}^{-1}$
its inverse. $X^{m,b}$ is the
completion of $\mathcal{S}(\mathbb{R} \times \mathbb{R}^3)$ with respect to
$$ 
\|f\|_{X^{m,b}}  =  \| \langle \xi \rangle^m \langle \tau \rangle^b
{\mathcal F}(e^{-it\Delta} f(x,t))\|_{L^2_{\xi, \tau}} = \| \langle \xi
\rangle^m \langle \tau + |\xi|^2 \rangle^b
\widehat{f}(\xi,\tau)\|_{L^2_{\xi,\tau}} \,, $$
For a given time interval $I$ we define
$$ 
\|f\|_{X^{m,b}(I)} := \inf_{g_{|I}=f}  \|g\|_{X^{m,b}} \,. 
$$

We recall the following facts about the solutions $u$ of the inhomogeneous
linear
Schr\"odinger equation (see e.g. \cite{GTV})
\begin{equation} \label {4'}
i u_t - \Delta u = F \,, \quad u(0) = f \,.
\end{equation}
For $ b'+1 \ge b \ge 0 \ge b' > -1/2 $ and $T \le 1$ we have
$$ 
\|u\|_{X^{m,b}[0,T]} \lesssim \|f\|_{H^m}  +
T^{1+b'-b}\|F\|_{X^{m,b'}[0,T]} \,. 
$$
For $ 1/2 > b > b' \ge 0$ or $0 \ge b > b' > -1/2$:
$$
 \|f\|_{X^{m,b'}[0,T]} \lesssim T^{b-b'} \|f\|_{X^{m,b}[0,T]}
 $$
(see e.g. \cite[ Lemma 1.10]{G}).

Fundamental are the following Strichartz type estimates for the solution $u$ of
\eqref{4'} in three space dimensions (see \cite{CH,KT}):
$$ 
\|u\|_{L^q(I,L^r(\mathbb{R}^3))} \lesssim \|f\|_{L^2(\mathbb{R}^3)} +
\|F\|_{L^{\tilde{q}'}(I,L^{\tilde{r}'}(\mathbb{R}^3))}
$$
with implicit constant independent of the interval $I \subset \mathbb{R}$
for all pairs $(q,r),(\tilde{q},\tilde{r})$ with 
$q,r,\tilde{q},\tilde{r} \ge 2$
and $\frac{1}{q} + \frac{3}{2r} = \frac{3}{4}$ , 
$\frac{1}{\tilde{q}} + \frac{3}{2\tilde{r}} = \frac{3}{4}$, where
$\frac{1}{\tilde{q}}+\frac{1}{\tilde{q}'} = 1$ and
$\frac{1}{\tilde{r}}+\frac{1}{\tilde{r}'} = 1$.
This implies
$$ 
\|\psi\|_{L^q(I,L^r(\mathbb{R}^3))} \lesssim
\|\psi\|_{X^{0,\frac{1}{2}+}(I)}  \,. 
$$

For real numbers $a$ we denote by $a+$, $a++$, $a-$ and $a--$ the numbers
$a+\epsilon$, $a+2\epsilon$, $a-\epsilon$ and $a-2\epsilon$, respectively,
where $\epsilon > 0$ is sufficiently small.
 We also use the notation $\langle x \rangle := (1+|x|^2)^{1/2}$
for $x \in \mathbb{R}^3$.

The paper is organized as follows: in section 1 we prove the uniqueness result 
Theorem \ref{Theorem 1.1} and two versions of a local well-posedness result 
for \eqref{4},\eqref{5},\eqref{6}, namely $u \in
X^{s,\frac{1}{2}+}[0,\delta] $ for data $u_0 \in H^s$ with $s \ge 0$ 
(Theorem \ref{Theorem 1.2}), and a modification
where $\nabla Iu \in X^{0, \frac{1}{2}+}[0,\delta]$ for data $\nabla Iu_0 \in
L^2$ (Proposition \ref{Prop}), which
is necessary in order to combine it with an almost conservation law for the
modified energy $E(Iu)$. In section 2 we use these local results and bounds
for the modified energy given in section 3 in order to get the main theorem 
(Theorem \ref{Theorem 1.3}).
Under the assumptions (A1) and (A2) it is namely shown that the bounds for
 the modified energy immediately give a polynomial bound for 
$\|\nabla Iu(t)\|_{L^2}$, which can be shown to imply a uniform exponential 
bound for $\|u(t)\|_{L^2}$, and as a consequence for $ \|u(t)\|_{H^s}$, 
which in view of the local well-posedness result suffices to get a global solution.
 Under the assumptions (A1) and (A3) we cannot immediately get a bound for  
$\|\nabla Iu(t)\|_{L^2}$ from the bound for the modified energy, but first we 
have to show an (exponential) bound for $\|Iu(t)\|_{L^2}$, which together 
with an energy bound gives the desired (exponential) bound for 
$\|\nabla Iu(t)\|_{L^2}$ and after that as in the previous case the 
global well-posedness result.
In section 3 we first calculate $\frac{d}{dt}
E(Iu)$ for any solution of the equation \eqref{4} and
estimate the time integrated terms which appear in  $\frac{d}{dt} E(Iu)$, 
which is the most complicated part. In section 2 these estimates are shown 
to control the modified energy $E(Iu)$
uniformly on arbitrary time intervals $[0,T]$, provided $s > 1/2$.

\section{Uniqueness and local well-posedness}

\begin{proof}[Proof of Theorem \ref{Theorem 1.1}]
Let $u,v \in C^0([0,T],L^2))$ be two solutions. Using Strichartz type estimates 
in order to control
$$ 
\|u-v\|_{L^2_t L_x^6} + \|u-v\|_{L^{\infty}_t L^2_x} 
$$
we have to estimate the various terms of $F(u)-F(v)$. By (A1) and the 
Hausdorff-Young inequality we have $W \in L^p$ for $1\le p < 3$ so that
 by Young's inequality we obtain
\begin{align*}
\|(W*|u|^2)(u-v)\|_{L^{1+}_t L^{2-}_x} 
& \lesssim \|W*|u|^2\|_{L^{2+}_t L^{3-}_x} \|u-v\|_{L^2_t L_x^6} \\
& \lesssim \|W\|_{L^{3-}_x} \|u\|_{L^{4+}_t L^2_x}^2 \|u-v\|_{L^2_t L^6_x} \\
&\lesssim T^{\frac{1}{2}-} \|u\|_{L^{\infty}_t L^2_x}^2 \|u-v\|_{L_t^2 L_x^6}\,,
\end{align*}
\begin{align*} 
\|(W*(|u|^2-|v|^2))v\|_{L^{1+}_t L^{2-}_x} 
& \lesssim \|W*(|u|^2 -|v|^2)\|_{L^{2}_t L^{\infty-}_x} \|u\|_{L^{2+}_t L_x^2} \\
& \lesssim \|W\|_{L^{3-}_x} \||u|^2 - |v|^2\|_{L^2_t L^{3/2}_x} \|u\|_{L^{2+}_t L^2_x} \\
&\lesssim T^{\frac{1}{2}-} (\|u\|_{L^{\infty}_t L^2_x} 
+ \|v\|_{L^{\infty}_t L^2_x}) \|u-v\|_{L_t^2 L_x^6} \|u\|_{L_t^{\infty} L_x^2}\,,
\end{align*}
\begin{align*}
 \|(W* \operatorname{Re} u)(u-v)\|_{L_t^1 L_x^2} 
& \lesssim \|W* \operatorname{Re} u\|_{L_t^1 L_x^{\infty}} 
 \|u-v\|_{L_t^{\infty} L_x^2} \\
& \lesssim \|W\|_{L^2} T \|u\|_{L_t^{\infty}L_x^2} \|u-v\|_{L_t^{\infty} L_x^2}\,,
\end{align*}
\[
 \|W* \operatorname{Re} (u-v)\|_{L_t^1 L_x^2}
 \lesssim \|W\|_{L^1} T  \|u-v\|_{L_t^{\infty} L_x^2} \,.
\]
Similarly the remaining terms can be estimated. Therefore, choosing $T$ small enough,
 we obtain $u \equiv v$.
\end{proof}

Next we prove the local well-posedness results.

\begin{proof}[Proof of Theorem \ref{Theorem 1.2}]
We want to apply the contraction mapping principle in the Bourgain type 
space $X^{s,\frac{1}{2}+}[0,\delta]$.
We have to estimate
$$ 
\|(W*(|u|^2 + 2 \operatorname{Re} u))(1+u)\|_{X^{s,-\frac{1}{2}++}} \,.
 $$
We denote by $D^l$ the operator with symbol $|\xi|^l$ and similarly by 
$\langle D \rangle^l$ the operator with symbol $\langle \xi \rangle^l$.
To estimate the cubic term by $\delta^{\frac{1}{2}+s-} \|u\|_{X^{s,\frac{1}{2}+}}^3$ 
we have to show (ignoring complex conjugates, which play no role for the calculations):
$$ 
\int \langle D \rangle^s ( \langle D \rangle^{-2}(u_1 u_2) u_3)\psi \,dx\,dt
 \lesssim \delta^{\frac{1}{2}+s-} \prod_{i=1}^3\|u_i\|_{X^{s,\frac{1}{2}+}} 
\|\psi\|_{X^{0,\frac{1}{2}--}} \,. 
$$
This is sufficient, because we can assume without loss of generality that the Fourier 
transforms $\widehat{u}_i(\xi_i,\tau_i)$ and $\widehat{\psi}(\xi,\tau)$ are 
nonnegative, so that using the fundamental assumption 
$|\widehat{W}(\xi)| \lesssim \langle \xi \rangle^{-2}$ it is possible to replace
 here and in similar situations in the following the convolution with $W$ by 
application of $\langle D \rangle^{-2}$.
Using the Leibniz rule for fractional derivatives we reduce to the estimates 
(assuming without loss of generality $|\xi_2| \ge |\xi_1|$):
$$ 
\|  \langle D \rangle^{-2}(u_1 \langle D \rangle^s u_2) u_3 \psi\|_{L^1_{xt}} 
\lesssim \delta^{\frac{1}{2}+s-} \prod_{i=1}^3 \|u_i\|_{X^{s,\frac{1}{2}+}} 
\|\psi\|_{X^{0,\frac{1}{2}--}}  
$$
and
$$
 \|   \langle D\rangle^{-2}(u_1 u_2) \langle D \rangle^s u_3\psi\|_{L^1_{xt}} 
\lesssim \delta^{\frac{1}{2}+s-} \prod_{i=1}^3 \|u_i\|_{X^{s,\frac{1}{2}+}} 
\|\psi\|_{X^{0,\frac{1}{2}--}} \,.
 $$
We obtain
$$ 
\|  \langle D \rangle^{-2}(u_1 \langle D \rangle^s u_2) u_3 \psi\|_{L^1_{xt}} 
\lesssim \|  \langle D \rangle^{-2}(u_1 \langle D \rangle^s u_2) u_3\|_{L^{1+}_t L^2_x} 
\|\psi\|_{L_t^{\infty -} L^2_x}  
$$
and
\begin{align*}
\| \langle D \rangle^{-2}(u_1 \langle D \rangle^s u_2)u_3 \|_{L^{1+}_t L^2_x} 
&\lesssim
\|\langle D \rangle^{-2}(u_1 \langle D \rangle^s u_2)\|_{L^{1+}_t L^{\hat{p}}_x} \|u_3\|_{L_t^{\infty} L_x^{\hat{q}}} \\
& \lesssim
 \| u_1 \langle D \rangle^s  u_2\|_{L^{1+}_t L^{\hat{r}}_x} \|u_3\|_{L_t^{\infty} H_x^s} \\
& \lesssim
\| u_1 \|_{L_t^{1+} L_x^{\hat{v}}} \|\langle D \rangle^s  u_2\|_{L^{\infty}_t L^2_x} \|u_3\|_{L_t^{\infty} H_x^s} \\
& \lesssim
\| u_1 \|_{L_t^{1+} H_x^{s,\hat{t}}} \|u_2\|_{X^{s,\frac{1}{2}+}} \|u_3\|_{X^{s,\frac{1}{2}+}} \\
& \lesssim \delta^{s+\frac{1}{2}-}
\| u_1 \|_{L_t^{\hat{w}} H_x^{s,\hat{t}}} \|u_2\|_{X^{s,\frac{1}{2}+}} \|u_3\|_{X^{s,\frac{1}{2}+}} \\
& \lesssim \delta^{s+\frac{1}{2}-} \prod_{i=1}^3\|u_i\|_{X^{s,\frac{1}{2}+}} \,.
\end{align*}
Here we set $\frac{1}{\hat{q}} = \frac{1}{2}-\frac{s}{3}$,
 $\frac{1}{\hat{p}} = \frac{s}{3}$, $ \frac{1}{\hat{r}} = \frac{2}{3}+\frac{s}{3}$, 
so that $H_x^s \subset L_x^{\hat{q}}$, and 
$\frac{1}{\hat{v}} = \frac{1}{6} +\frac{s}{3}$,  
$\frac{1}{\hat{t}} = \frac{1}{6} + \frac{2}{3}s$ , so that 
$H^{s,\hat{t}}_x \subset L_x^{\hat{v}}$, and finally 
$\frac{1}{\hat{w}} = \frac{1}{2}-s$, so that 
$ \frac{1}{\hat{w}} + \frac{3}{2\hat{t}} = \frac{3}{4}$, 
which allows to apply Strichartz' estimate in the last line. Similarly we obtain
\begin{align*}
\| \langle D \rangle^{-2}(u_1 u_2) \langle D \rangle^s u_3 \|_{L^{1+}_t L^2_x} &\lesssim
\|\langle D \rangle^{-2}(u_1 u_2)\|_{L^{1+}_t L^{\infty}_x} \|\langle D \rangle^s u_3\|_{L_t^{\infty} L_x^2} \\
 & \lesssim
 \| u_1 u_2\|_{L^{1+}_t L^{\frac{3}{2}+}_x} \|u_3\|_{X^{s,\frac{1}{2}+}} \\
& \lesssim
\| u_1 \|_{L_t^{1+} L_x^{\hat{v}}}\| u_2 \|_{L_t^{\infty} L_x^{\hat{q}}} \|u_3\|_{X^{s,\frac{1}{2}+}} \\
& \lesssim \delta^{s+\frac{1}{2}-}
\| u_1 \|_{L_t^{\hat{w}} H_x^{s,\hat{t}}} \| u_2 \|_{L_t^{\infty} H_x^{s}} 
\|u_3\|_{X^{s,\frac{1}{2}+}} \\
& \lesssim \delta^{s+\frac{1}{2}-} \prod_{i=1}^3 \|u_i\|_{X^{s,\frac{1}{2}+}} \,.
\end{align*}
Here $\frac{1}{\hat{q}} = \frac{1}{2}-\frac{s}{3}$, 
$\frac{1}{\hat{v}} = \frac{1}{6} +\frac{s}{3}-$,
$\frac{1}{\hat{t}} = \frac{1}{6} + \frac{2}{3}s-$, 
$\frac{1}{\hat{w}} = \frac{1}{2}-s+$, so that 
$\frac{1}{\hat{w}} + \frac{3}{2\hat{t}} = \frac{3}{4}$ allows to apply Strichartz' 
estimate again.

The quadratic terms are handled as follows (assuming again $|\xi_2| \ge |\xi_1|$):
\begin{align*}
\|\langle D \rangle^{-2+s} (u_1 u_2)\|_{L_t^{1+} L_x^2}
 & \lesssim \|\langle D \rangle^s (u_1 u_2)\|_{L_t^{1+} L_x^1} 
\lesssim \|u_1 \langle D \rangle^s u_2\|_{L_t^{1+}L_x^1} \\ 
& \lesssim \|u_1\|_{L_t^{2+} L_x^2} \| \langle D \rangle^s u_2 \|_{L_t^2 L_x^2} 
\lesssim \delta^{1-} \|u_1\|_{X^{s,\frac{1}{2}+}}  \|u_2\|_{X^{s,\frac{1}{2}+}}
\end{align*}
and
\begin{align*}
&\|\langle D \rangle^s ((\langle D \rangle^{-2}u_1)u_2)\|_{L_t^{1+} L_x^2} \\ & \lesssim \|(\langle D \rangle^{-2+s} u_1)u_2\|_{L_t^{1+} L_x^2} + \|\langle D \rangle^{-2} u_1 \langle D \rangle^s u_2 \|_{L_t^{1+} L_x^2} \\
&\lesssim \| \langle D \rangle^{-2+s} u_1\|_{L_t^{1+}L_x^{\infty}} \|u_2\|_{L_t^{\infty} L_x^2} + \| \langle D \rangle^{-2} u_1\|_{L_{t,x}^{\infty}} \|\langle D \rangle^s u_2\|_{L_t^{1+} L_x^2}\\
& \lesssim \delta^{1-} \| \langle D \rangle^s u_1\|_{L_t^{\infty}L_x^2} \|u_2\|_{L_t^{\infty} L_x^2} + \delta^{1-}\| u_1\|_{L_t^{\infty}L_x^2} \|\langle D \rangle^s u_2\|_{L_t^{\infty} L_x^2}\\
& \lesssim \delta^{1-} \|u_1\|_{X^{s,\frac{1}{2}+}} \|u_2\|_{X^{s,\frac{1}{2}+}} \,.
\end{align*}
Similar estimates hold for the difference $F(u)-F(v)$, so that a
 standard Picard iteration under the conditions 
$ \delta^{s+\frac{1}{2}-} \|u_0\|_{H^s}^2 \lesssim 1$ and 
$\delta^{1-} \|u_0\|_{H^s} \lesssim 1$ shows the existence of a unique solution 
in $X^{s,\frac{1}{2}+}[0,\delta] \subset C^0([0,\delta],H^s)$.
 It is also unique is this latter space by Theorem \ref{Theorem 1.1}.
\end{proof}

\noindent \textbf{Remark:} This Theorem shows that in order to get a 
global solution it is sufficient to give an a-priori bound of $\|u(t)\|_{H^s}$.

We next prove a modified local well-posedness result involving the operator
 $I$ (recall that $I$ depends on $s$ and $N$).

\begin{proposition} \label{Prop}
Assume $s \ge 0$ and $\nabla Iu_0 \in L^2$. Then the Cauchy problem 
\eqref{4},\eqref{5},\eqref{6} (after application of $I$) has a unique local 
solution $u$ with $\nabla Iu \in X^{0,\frac{1}{2}+}[0,\delta]$ and 
$\|\nabla Iu\|_{X^{0,\frac{1}{2}+}[0,\delta]} \le \sqrt{M} \|\nabla Iu_0\|_{L^2}$,
 where $M\ge 1$ is independent of $u_0$, and $\delta \le 1$ can be chosen such that
$$ 
(\delta^{\frac{1}{2}-}N^{-2} + \delta^{1-}) \|\nabla Iu_0\|^2_{L^2} \sim 1 \,. 
$$
\end{proposition}

\begin{proof}
The cubic term in the nonlinearity will be estimated as follows 
(dropping $[0,\delta]$ from the notation):
$$
\| \nabla I(W* (u_1 u_2)u_3)\|_{X^{0,-\frac{1}{2}++}} \lesssim
(\frac{\delta^{\frac{1}{2}-}}{N^2} + \delta^{1-}) 
\prod_{i=1}^3 \|\nabla Iu_i\|_{X^{0,\frac{1}{2}+}} \,. 
$$
This follows from
\begin{align*}
A & := \int_0^{\delta} \int_* M(\xi_1,\xi_2,\xi_3) 
 \prod_{i=1}^3 \widehat{u}_i(\xi_i,t) \widehat{\psi}(\xi_4,t) \,d\xi_1 \,d\xi_2 \,d\xi_3 \,d\xi_4 dt \\
& \lesssim (\frac{\delta^{\frac{1}{2}-}}{N^2} 
+ \delta^{1-}) \prod_{i=1}^3 \|u_i\|_{X^{0,\frac{1}{2}+}}\|\psi\|_{X^{0,\frac{1}{2}+}} \,,
\end{align*}
where
$$
 M(\xi_1,\xi_2,\xi_3) := \frac{m(\xi_1+\xi_2+\xi_3)}{m(\xi_1)m(\xi_2)m(\xi_3)}
 \cdot \frac{|\xi_1 + \xi_2 + \xi_3|}{\langle \xi_1 
+ \xi_2 \rangle^2 |\xi_1||\xi_2||\xi_3|} \,, $$
and * denotes integration over the region $\{\sum_{i=1}^4 \xi_i = 0\}$. 
We assume here and in the following again  without loss of generality
 that the Fourier transforms are nonnegative, and also without loss of generality
 that $|\xi_1| \ge |\xi_2|$. We again used the property 
$|\widehat{W}(\xi)| \lesssim \langle\xi \rangle^{-2}$.

We make a case by case analysis depending on the relative size of the frequencies. 

\noindent\textbf{Case 1:} $|\xi_1| \ge |\xi_2| \ge |\xi_3| \gtrsim N$. \\
\textbf{1a.} $|\xi_1 + \xi_2 + \xi_3| \ge N$. 
We obtain
\begin{align*}
M(\xi_1,\xi_2,\xi_3) 
& \lesssim \prod_{i=1}^3 (\frac{|\xi_i|}{N})^{1-s} \frac{N^{1-s}}{|\xi_1+\xi_2+\xi_3|^{1-s}} \frac{|\xi_1+\xi_2+\xi_3|}{|\xi_1||\xi_2||\xi_3| \langle \xi_1 + \xi_2 \rangle^2} \\
&\lesssim \frac{1}{N^{2(1-s)}} \frac{|\xi_1|^s}{|\xi_1|^s|\xi_2|^s|\xi_3|^s \langle \xi_1 + \xi_2 \rangle^2} \lesssim \frac{1}{N^2 \langle \xi_1 + \xi_2 \rangle^2} \,.
\end{align*}
This implies by Sobolev's embedding and Strichartz' estimates:
\begin{align*}
A
& \lesssim \frac{1}{N^2} \| \langle D \rangle^{-2}(u_1 u_2)\|_{L_t^{\infty} L_x^3} \|u_3\|_{L_t^2 L_x^6} \|\psi\|_{L_t^2 L_x^2} \\
& \lesssim \frac{1}{N^2} \|u_1 u_2\|_{L_t^{\infty} L_x^1} \|u_3\|_{L_t^2 L_x^6} \delta^{\frac{1}{2}-} \|\psi\|_{L_t^{\infty -} L_x^2} \\
& \lesssim \frac{1}{N^2} \delta^{\frac{1}{2}-} \prod_{i=1}^3 \|u_i\|_{X^{0,\frac{1}{2}+}} \|\psi\|_{X^{0,\frac{1}{2}--}} \,.
\end{align*}
\textbf{1b.} $|\xi_1+\xi_2+\xi_3| \le N$. 
We have
\begin{align*}
M(\xi_1,\xi_2,\xi_3) & \lesssim \prod_{i=1}^3 (\frac{|\xi_i|}{N})^{1-s} \frac{N}{|\xi_1||\xi_2||\xi_3| \langle \xi_1 + \xi_2 \rangle^2} \\
&\lesssim \frac{N}{N^{3(1-s)}|\xi_1|^s|\xi_2|^s|\xi_3|^s \langle \xi_1 + \xi_2 \rangle^2} \lesssim \frac{1}{N^2 \langle \xi_1 + \xi_2 \rangle^2}
\end{align*}
as in case 1a. 

\noindent\textbf{Case 2:} $|\xi_3| \ge |\xi_1| \ge |\xi_2|$. 
This case can be treated similarly as case 1.

\noindent\textbf{Case 3:} $|\xi_1| \ge |\xi_2| \gtrsim N \ge |\xi_3|$. \\
\textbf{3a.} $|\xi_1 + \xi_2 + \xi_3| \ge N$. 
We obtain
\begin{align*}
M(\xi_1,\xi_2,\xi_3) 
& \lesssim \frac{|\xi_1|^{1-s}}{N^{1-s}} \frac{|\xi_2|^{1-s}}{N^{1-s}} \frac{N^{1-s}}{|\xi_1+\xi_2 + \xi_3|^{1-s}}
\frac{|\xi_1+\xi_2+\xi_3|}{|\xi_1||\xi_2||\xi_3| \langle \xi_1 + \xi_2 \rangle^2} \\
& \lesssim \frac{|\xi_1|^s}{|\xi_1|^s |\xi_2|^s |\xi_3| N^{1-s} \langle \xi + \xi_2 \rangle^2} \lesssim \frac{1}{N |\xi_3| \langle \xi_1 + \xi_2 \rangle^2} \,,
\end{align*}
so that
\begin{align*}
A 
& \lesssim \frac{1}{N} \|\langle D \rangle^{-2}(u_1 u_2)\|_{L_t^{\infty} L_x^3} \|D^{-1} u_3\|_{L_t^{\infty} L_x^6} \|\psi\|_{L_t^1 L_x^2} \\
& \lesssim \frac{1}{N} \delta^{1-} \|u_1 u_2\|_{L_t^{\infty} L_x^1} \| u_3\|_{L_t^{\infty} L_x^2} \|\psi\|_{L_t^{\infty -} L_x^2} \\
& \lesssim \frac{1}{N} \delta^{1-}  \prod_{i=1}^3 \|u_i\|_{X^{0,\frac{1}{2}+}} \|\psi\|_{X^{0,\frac{1}{2}--}} \,.
\end{align*}
\textbf{3b.} $|\xi_1 + \xi_2 + \xi_3| \le N$. 
Similarly we obtain
$$
M(\xi_1,\xi_2,\xi_3)  \lesssim \frac{|\xi_1|^{1-s}}{N^{1-s}} \frac{|\xi_2|^{1-s}}{N^{1-s}}
\frac{N}{|\xi_1||\xi_2||\xi_3| \langle \xi_1 + \xi_2 \rangle^2}
  \lesssim \frac{1}{N |\xi_3| \langle \xi_1 + \xi_2 \rangle^2}
$$
as in case 3a.

\noindent\textbf{Case 4:} $|\xi_1|,|\xi_3| \gtrsim N \gtrsim |\xi_2|$.
Similarly as in case 3 we obtain
$$ 
M(\xi_1,\xi_2,\xi_3) \lesssim \frac{1}{N |\xi_2| \langle \xi_1 + \xi_2 \rangle^2} \,, 
$$
so that by Sobolev's embedding and Strichartz' estimates
\begin{align*}
A & \lesssim \frac{1}{N} \|\langle D \rangle^{-2}(u_1 D^{-1} u_2)\|_{L_t^{\infty -} L_x^{\infty}} \| u_3\|_{L_t^{\infty} L_x^2} \|\psi\|_{L_t^{1+} L_x^2} \\
& \lesssim \frac{1}{N} \delta^{1-} \|u_1 D^{-1} u_2\|_{L_t^{\infty -} L_x^{\frac{3}{2}+}} \| u_3\|_{L_t^{\infty} L_x^2} \|\psi\|_{L_t^{\infty -} L_x^2} \\
& \lesssim \frac{1}{N} \delta^{1-} \|u_1\|_{L_t^{\infty -} L_x^{2+}} \| D^{-1} u_2\|_{L_t^{\infty} L_x^6} \| u_3\|_{L_t^{\infty} L_x^2} \|\psi\|_{L_t^{\infty -} L_x^2} \\
& \lesssim \frac{1}{N} \delta^{1-}  \prod_{i=1}^3 \|u_i\|_{X^{0,\frac{1}{2}+}} \|\psi\|_{X^{0,\frac{1}{2}--}} \,.
\end{align*}

\noindent\textbf{Case 5:} $ |\xi_3| \gtrsim N \gg |\xi_1|\ge|\xi_2|$. 
We have
\begin{align*}
M(\xi_1,\xi_2,\xi_3) & \lesssim  \frac{|\xi_3|^{1-s}}{N^{1-s}} \frac{N^{1-s}}{|\xi_1+\xi_2 + \xi_3|^{1-s}}
\frac{|\xi_1+\xi_2+\xi_3|}{|\xi_1||\xi_2||\xi_3| \langle \xi_1 + \xi_2 \rangle^2} \\
& \lesssim \frac{1}{|\xi_1| |\xi_2| \langle \xi_1 + \xi_2 \rangle^2}  \,,
\end{align*}
leading to
\begin{align*}
A & \lesssim \|\langle D \rangle^{-2}(D^{-1}u_1 D^{-1} u_2)\|_{L_t^{\infty} L_x^{\infty}} \| u_3\|_{L_t^{\infty} L_x^2} \|\psi\|_{L_t^{1} L_x^2} \\
& \lesssim \delta^{1-} \|D^{-1}u_1 D^{-1} u_2\|_{L_t^{\infty} L_x^{3}} \| u_3\|_{L_t^{\infty} L_x^2} \|\psi\|_{L_t^{\infty -} L_x^2} \\
& \lesssim  \delta^{1-} \|D^{-1}u_1\|_{L_t^{\infty} L_x^6} \| D^{-1} u_2\|_{L_t^{\infty} L_x^6} \| u_3\|_{L_t^{\infty} L_x^2} \|\psi\|_{L_t^{\infty -} L_x^2} \\
& \lesssim  \delta^{1-}  \prod_{i=1}^3 \|u_i\|_{X^{0,\frac{1}{2}+}} \|\psi\|_{X^{0,\frac{1}{2}--}} \,.
\end{align*}

\noindent\textbf{Case 6:} $|\xi_1| \gtrsim N \gg |\xi_2|,|\xi_3|$. 
Similarly as in case 5 we obtain
$$
M(\xi_1,\xi_2,\xi_3)  \lesssim \frac{1}{|\xi_2| |\xi_3| \langle \xi_1
 + \xi_2 \rangle^2}  \,,
$$
which implies
\begin{align*}
A & \lesssim \delta^{1-} \|\langle D \rangle^{-2}(u_1 D^{-1} u_2)\|_{L_t^{\infty} L_x^3} \|D^{-1} u_3\|_{L_t^{\infty} L_x^6} \|\psi\|_{L_t^{\infty -} L_x^2} \\
& \lesssim \delta^{1-} \|u_1 D^{-1} u_2\|_{L_t^{\infty -} L_x^{3/2}} \| u_3\|_{L_t^{\infty} L_x^2} \|\psi\|_{L_t^{\infty -} L_x^2} \\
& \lesssim  \delta^{1-} \|u_1\|_{L_t^{\infty} L_x^2} \| D^{-1} u_2\|_{L_t^{\infty} L_x^6} \| u_3\|_{L_t^{\infty} L_x^2} \|\psi\|_{L_t^{\infty -} L_x^2} \\
& \lesssim  \delta^{1-}  \prod_{i=1}^3 \|u_i\|_{X^{0,\frac{1}{2}+}} \|\psi\|_{X^{0,\frac{1}{2}--}} \,.
\end{align*}

\noindent \textbf{Case 7:} $N \gg |\xi_1|,|\xi_2|,|\xi_3|$. 
We easily obtain
$$
M(\xi_1,\xi_2,\xi_3)  \lesssim \frac{|\xi_1+\xi_2+\xi_3|}{|\xi_1||\xi_2||\xi_3|
 \langle \xi_1 + \xi_2 \rangle^2} \lesssim \frac{1}{|\xi_2| |\xi_3| \langle \xi_1 
+ \xi_2 \rangle^2} \,\,\,{\mbox or}\, \lesssim \frac{1}{|\xi_1| |\xi_2| \langle \xi_1 
+ \xi_2 \rangle^2}\,, 
$$
which can be handled like case 6 or case 5. This completes the claimed estimate 
for the cubic term.

Next we consider the quadratic terms in the nonlinearity. 
They turn out to be less critical. First we prove the estimate
$$ 
\|\nabla I W* (u_1 u_2)\|_{X^{0,-\frac{1}{2}++}} \lesssim
(\frac{\delta^{\frac{1}{2}-}}{N} 
+ \delta^{1-})\|\nabla Iu_1\|_{X^{0,\frac{1}{2}+}} 
\|\nabla Iu_2\|_{X^{0,\frac{1}{2}+}} \,. 
$$
This follows from
\begin{align*}
B & := \int_0^{\delta} \int_* M(\xi_1,\xi_2) \prod_{i=1}^2 \widehat{u}_i(\xi_i,t) \widehat{\psi}(\xi_3,t) \,d\xi_1 \,d\xi_2 \,d\xi_3 dt \\
& \lesssim (\frac{\delta^{\frac{1}{2}-}}{N} + \delta^{1-}) \prod_{i=1}^2 \|u_i\|_{X^{0,\frac{1}{2}+}}\|\psi\|_{X^{0,\frac{1}{2}+}} \,,
\end{align*}
where
$$
 M(\xi_1,\xi_2) := \frac{m(\xi_1+\xi_2)}{m(\xi_1)m(\xi_2)} \cdot 
\frac{|\xi_1 + \xi_2|}{\langle \xi_1 + \xi_2 \rangle^2 |\xi_1||\xi_2|} \,, 
$$
and * denotes integration over the region $\{\sum_{i=1}^3 \xi_i = 0\}$. 
We assume without loss of generality $|\xi_1| \ge |\xi_2|$. 

\noindent\textbf{Case 1:} $|\xi_1| \ge |\xi_2| \ge N$. \\
\textbf{1a.} $|\xi_1 + \xi_2| \ge N$. 
We obtain
\begin{align*}
M(\xi_1,\xi_2) & \lesssim \frac{|\xi_1|^{1-s}}{N^{1-s}} \frac{|\xi_2|^{1-s}}{N^{1-s}} \frac{N^{1-s}}{|\xi_1+\xi_2|^{1-s}}
\frac{|\xi_1+\xi_2|}{|\xi_1||\xi_2| \langle \xi_1 + \xi_2 \rangle^2} \\
& \lesssim \frac{|\xi_1 + \xi_2|^s}{|\xi_1|^s |\xi_2|^s N^{1-s} \langle \xi_1 + \xi_2 \rangle^2} \lesssim \frac{1}{N \langle \xi_1 + \xi_2 \rangle^2} \,,
\end{align*}
so that by Strichartz' estimates
\begin{align*}
B & \lesssim \frac{1}{N} \|\langle D \rangle^{-2}(u_1 u_2)\|_{L_t^2 L_x^2} \|\psi\|_{L_t^2 L_x^2} \\
& \lesssim \frac{1}{N} \|u_1 u_2\|_{L_t^2 L_x^{3/2}} \|\psi\|_{L_t^2 L_x^2} \\
& \lesssim \frac{1}{N} \delta^{\frac{1}{2}-} \|u_1\|_{L_t^{\infty} L_x^2} \|u_2\|_{L_t^2 L_x^6} \|\psi\|_{L_t^{\infty -} L_x^2} \\
& \lesssim \frac{1}{N} \delta^{\frac{1}{2}-} \|u_1\|_{X^{0,\frac{1}{2}+}} \|u_2\|_{X^{0,\frac{1}{2}+}} \|\psi\|_{X^{0,\frac{1}{2}--}} \,.
\end{align*}
\textbf{1b.} $|\xi_1 + \xi_2| \le N$.
This case is similar to case 1a. 

\noindent\textbf{Case 2:} $ N \gtrsim |\xi_2|$ and $|\xi_1| \gg |\xi_2|$. 
One has
$$
 M(\xi_1,\xi_2) \lesssim \frac{|\xi_1 + \xi_2|}{\langle \xi_1
 + \xi_2 \rangle^2 |\xi_1||\xi_2|} \lesssim \frac{1}{|\xi_2| \langle \xi_1
 + \xi_2 \rangle^2} \,, 
$$
so that
\begin{align*}
B & \lesssim \|\langle D \rangle^{-2}(u_1 D^{-1}u_2)\|_{L_t^2 L_x^2} \|\psi\|_{L_t^2 L_x^2} \\
& \lesssim \|u_1 D^{-1}u_2\|_{L_t^2 L_x^{3/2}} \|\psi\|_{L_t^2 L_x^2} \\
& \lesssim \delta^{1-} \|u_1\|_{L_t^{\infty} L_x^2} \|D^{-1}u_2\|_{L_t^{\infty} L_x^6} \|\psi\|_{L_t^{\infty -} L_x^2} \\
& \lesssim \delta^{1-} \|u_1\|_{X^{0,\frac{1}{2}+}} \|u_2\|_{X^{0,\frac{1}{2}+}} \|\psi\|_{X^{0,\frac{1}{2}--}} \,.
\end{align*}

\noindent\textbf{Case 3:} $ N \ge |\xi_1| \sim |\xi_2|$.
This case can be handled like case 2.

Finally we show
$$ 
\|\nabla I((W*u_1) u_2)\|_{X^{0,-\frac{1}{2}++}} \lesssim
 \delta^{1-}\|\nabla Iu_1\|_{X^{0,\frac{1}{2}+}} 
\|\nabla Iu_2\|_{X^{0,\frac{1}{2}+}} \,.
 $$
Here $B$ is as in case 2 with
$$
 M(\xi_1,\xi_2) := \frac{m(\xi_1+\xi_2)}{m(\xi_1)m(\xi_2)} \cdot \frac{|\xi_1 
+ \xi_2|}{\langle \xi_1  \rangle^2 |\xi_1||\xi_2|} \,. 
$$
Because the estimates are similar to the previous case we only consider
 the most critical low frequency cases. 

\noindent\textbf{Case 1:} $|\xi_1| \gg |\xi_2|$ and $N \ge |\xi_1|$ 
(or $N \gg |\xi_1| \sim |\xi_2|$). 
The estimate
$$
 M(\xi_1,\xi_2) \lesssim \frac{1}{\langle \xi_1 \rangle^2 |\xi_1|}
 $$
implies
\begin{align*}
B & \lesssim \|\langle D \rangle^{-2}D^{-1}u_1\|_{L_{x,t}^{\infty}} \|u_2\|_{L_t^2 L_x^2} \|\psi\|_{L_t^2 L_x^2} \\
& \lesssim \|u_1\|_{L_t^{\infty} L_x^2} \|u_2\|_{L_t^2 L_x^2} \|\psi\|_{L_t^2 L_x^2}  \\
& \lesssim \delta^{1-} \|u_1\|_{X^{0,\frac{1}{2}+}} \|u_2\|_{X^{0,\frac{1}{2}+}} \|\psi\|_{X^{0,\frac{1}{2}--}} \,.
\end{align*}

\noindent\textbf{Case 2:} $|\xi_1| \gg |\xi_2|$ and $N \ge |\xi_2|$. 
The estimate
$$ 
M(\xi_1,\xi_2) \lesssim \frac{1}{\langle \xi_1 \rangle^2 |\xi_2|}
 $$
implies
\begin{align*}
B & \lesssim \|\langle D \rangle^{-2}u_1\|_{L_t^{\infty} L_x^3} \|D^{-1}u_2\|_{L_t^2 L_x^6} \|\psi\|_{L_t^2 L_x^2} \\
& \lesssim \delta^{1-} \|u_1\|_{L_t^{\infty} L_x^2} \|u_2\|_{L_t^{\infty} L_x^2} \|\psi\|_{L_t^{\infty -} L_x^2}  \\
& \lesssim \delta^{1-} \|u_1\|_{X^{0,\frac{1}{2}+}} \|u_2\|_{X^{0,\frac{1}{2}+}} \|\psi\|_{X^{0,\frac{1}{2}--}} \,.
\end{align*}

We remark that similar estimates can be given for the difference terms in order 
to use Banach's fixed point theorem. In order to get a contraction we have 
to fulfill the estimates
$$ 
(\delta^{\frac{1}{2}-}N^{-2} + \delta^{1-}) \|\nabla Iu_0\|_{L^2}^2 \ll 1 \quad
\text{and} \quad (\delta^{\frac{1}{2}-}N^{-1} 
+ \delta^{1-}) \|\nabla Iu_0\|_{L^2} \ll 1 \,. 
$$
The latter requirement is weaker, so that the claimed result follows.
\end{proof}

\noindent\textbf{Remark:} We want to iterate this local existence theorem 
with time steps of equal length until we reach a given (large) time $T$.
 To achieve this we need to control
\begin{equation} \label{16}
\|\nabla Iu(t)\|_{L^2} \le c(T)  \quad \forall \, 0 \le t \le T \,.
\end{equation}
This will be shown under the assumption $u_0 \in H^s$ with $ s > 1/2 $.

\section{Proof of Theorem \ref{Theorem 1.3}}

In this section we first show that the bound \eqref{16} implies global
 well-posedness and after that we derive such a bound from the estimates 
for the modified energy $E(Iu)$ in the next section.

\begin{proof}[Proof of Theorem \ref{Theorem 1.3}]
So let us assume for the moment that \eqref{16} holds. This means that on 
any existence interval $[0,T]$ we have an a-priori bound (for fixed $N$) of
\begin{equation}\label{17}
 \|\nabla Iu(t)\|_{L^2} \sim \| |\xi| \widehat{u}(\xi,t)\|_{L^2(\{|\xi| \le N\})} 
+ \||\xi|^s \widehat{u}(\xi,t)\|_{L^2(\{|\xi| \ge N\})} N^{1-s} \,,
 \end{equation}
especially
\begin{equation} \label{17'}
\| |\xi|^s \widehat{u}(\xi,t)\|_{L^2(\{|\xi| \ge 1\})} \le c(T) \,.
\end{equation}
If we can show that this implies an a-priori bound for $\|u(t)\|_{L^2}$,
 which is done in the following lemma, we immediately get an a-priori
 bound for $\|u(t)\|_{H^s}$ , $0 \le t \le T$ , thus a unique global solution
in $X^{s,\frac{1}{2}+}[0,T] \subset C^0([0,T],H^s)$ for any $T$ using our 
local well-posedness result (Theorem \ref{Theorem 1.2}), which is also unique 
in this latter space by Theorem \ref{Theorem 1.1}.
\end{proof}

\begin{lemma}
Assume \eqref{16} and $s \ge 1/2$. On any existence interval $[0,T]$ 
of the solution $u \in X^{s,\frac{1}{2}+}[0,T]$ we have $ \|u(t)\|_{L^2} \le c(T)$.
\end{lemma}

\begin{proof}
We  decompose $\widehat{u} = \widehat{u_1} + \widehat{u_2}$ smoothly with
 $\operatorname{supp} \widehat{u_1} \subset \{ |\xi| \le 2 \}$ and 
$ \operatorname{supp}  \widehat{u_2} \subset\{|\xi| \ge 1\}$. 
Then we have by Gagliardo-Nirenberg
\begin{align*}
\|u\|_{L^3} &\leq   \|u_1\|_{L^3} + \|u_2\|_{L^3} \lesssim \|\nabla
u_1\|_{L^2}^{1/2} \|u_1\|_{L^2}^{1/2} + \| |D|^{1/2}
u_2\|_{L^2} \\
& \lesssim   \|\nabla u_1\|_{L^2}^{1/2} \|u_1\|_{L^2}^{1/2} +
\| |D|^s u_2\|_{L^2}^{\frac{1}{2s}} \|u_2\|_{L^2}^{1-\frac{1}{2s}} \\
& \lesssim   \|\nabla u_1\|_{L^2}^2 + \|u_1\|_{L^2}^{\frac{2}{3}} +
\|u_2\|_{L^2}^{\frac{2}{3}}+ \| |D|^s u_2\|_{L^2}^{\frac{2}{3-2s}} \,,
\end{align*}
so that by \eqref{16},\eqref{17} and \eqref{17'} we obtain on $[0,T]$:
\begin{align*}
\|u(t)\|_{L^3}^3 
& \lesssim  \||\xi| \widehat{u_1}(\xi,t)\|_{L^2}^6 +
\|u_1(t)\|_{L^2}^2 + \|u_2(t)\|_{L^2}^2 + \||\xi|^s
\widehat{u_2}(\xi,t)\|_{L^2}^{\frac{6}{3-2s}}\\
&\leq   c'(T)( \|u(t)\|_{L^2}^2 +1) \,.
\end{align*}
Multiplying the differential equation \eqref{4} with $iu$ and taking the 
real part we obtain by Young's inequality, because $W \in L^1$ is real-valued:
\begin{align*}
\frac{d}{dt} \|u(t)\|_{L^2}^2 
& = \int (W*(|u|^2 + 2 \operatorname{Re} u))u dx + \operatorname{Re} i \int(W*(|u|^2 + 2 \operatorname{Re} u))|u|^2 dx \\
& \lesssim \int |W*(|u|^2)||u| dx + 2\int|W*\operatorname{Re} u||u| dx \\
& \lesssim \|W*|u|^2\|_{L^{3/2}} \|u\|_{L^3} + \|u\|_{L^2}^2 \\
&\lesssim \|u\|_{L^3}^3 + \|u\|_{L^2}^2 \\
& \le c'(T) (\|u(t)\|_{L^2}^2+1) \,,
\end{align*}
so that Gronwall's lemma gives
$$ 
\|u(t)\|_{L^2}^2 + 1  \le (\|u_0\|_{L^2}^2 +1) e^{c'(T)T} 
$$
on $[0,T]$.
\end{proof}

We recall our aim to give an a-priori bound of $\|\nabla Iu(t)\|_{L^2}$ 
(cf. \eqref{16}) on $[0,T]$ for an arbitrarily given $T$. 
We want to show this in the rest of this section as a consequence of 
Proposition \ref{Prop} and the estimates for the modified energy which
 we give in the next section.

Let $N\ge 1$ be a number to be specified later and $s > 1/2$. 
Let data $u_0 \in H^s$ be given. Then we have
\begin{align}  \label{28}
\|\nabla Iu_0\|_{L^2}^2 & \lesssim \| |\xi|\widehat{u_0}(\xi)\|_{L^2(\{|\xi|
\le N\})}^2 
+ \|N^{1-s} |\xi|^s \widehat{u_0}(\xi)\|_{L^2(\{|\xi| \ge N\})}^2  \\
& \lesssim   \|N^{1-s} |\xi|^s \widehat{u_0}(\xi)\|_{L^2(\mathbb{R}^3)}^2 =
N^{2(1-s)} \|u_0\|_{\dot{H}^s}^2 \lesssim N^{2(1-s)} \,.
\end{align}
This implies an estimate for $|E(Iu_0)|$ as follows: we have by Young's inequality, using $W \in L^1$:
$$ |\int (W*(|Iu_0|^2 + 2 \operatorname{Re} Iu_0))(|Iu_0|^2 + 2 \operatorname{Re} Iu_0) dx| \lesssim \|Iu_0\|_{L^4}^4 + \|Iu_0\|_{L^3}^3 + \|Iu_0\|_{L^2}^2 \,. $$
Now by Sobolev's embedding
\begin{align*}
&\|Iu_0\|_{L^4}^4  \\
&\lesssim \|Iu_0\|^4_{\dot{H}^{\frac{3}{4}}} 
 \lesssim \| |\xi|^{\frac{3}{4}} \widehat{u}_0\|^4_{L^2(\{|\xi| \le N\})}
  + \| \frac{N^{1-s}}{|\xi|^{1-s}} |\xi|^{\frac{3}{4}}
  \widehat{u}_0\|^4_{L^2(\{|\xi| \ge N\})} \\
& \lesssim  \|\widehat{u}_0\|^4_{L^2(\{|\xi| \le 1\})} 
 + \| |\xi|^{\frac{3}{4}-s} |\xi|^s \widehat{u}_0\|^4_{L^2(\{1 \le |\xi| \le N\})} 
 +  N^{3-4s} \||\xi|^s  \widehat{u}_0\|^4_{L^2(\{|\xi| \ge N\})} \\
& \lesssim \|u_0\|_{L^2}^4 + \langle N \rangle^{3-4s} \|u_0\|_{\dot{H}^s}^4 
 \lesssim N^{2(1-s)} \|u_0\|_{H^s}^4 \,,
\end{align*}
using in the last line the assumption $s\ge 1/2$.
Moreover
\begin{gather*}
 \|Iu_0\|^3_{L^3} \lesssim \|Iu_0\|^3_{\dot{H}^{1/2}}
 \le \|u_0\|^3_{\dot{H}^{1/2}} \le\|u_0\|_{H^s}^3,\\
\|Iu_0\|^2_{L^2} \le \|u_0\|_{L^2}^2 \,, 
\end{gather*}
so that
$$ 
|E(Iu_0)| \le c_0 N^{2(1-s)} \,.
 $$
The local existence theorem (Prop. \ref{Prop}) implies that there exists a solution $u$ on some time interval $[0,\delta]$ with
\begin{equation} \label{29}
 \|\nabla Iu\|^2_{X^{0,\frac{1}{2}+}[0,\delta]} 
\le M \|\nabla Iu_0\|_{L^2}^2 \le c_0 M N^{2(1-s)+2\epsilon}
\end{equation}
under the assumption $\|\nabla Iu_0\|_{L^2}^2 \le c_0 N^{2(1-s+\epsilon)}$,
 where $\epsilon \ge 0$ and 
\begin{equation} \label{29''}
\delta \sim \frac{1}{N^{2(1-s+\epsilon)}} \,.
\end{equation}
Now we use the results of the next section. We have the following estimate
\begin{equation} \label{29'}
\begin{aligned}
& |E(Iu(\delta))-E(Iu_0)| \\
&\lesssim  (\frac{\delta^{1/2}}{N^{2-}} + \frac{\delta^{1-}}{N^{1-}}
 + \frac{1}{N^{3-}}) \|\nabla Iu\|^4_{X^{0,\frac{1}{2}+}[0,\delta]}
 + \frac{\delta^{1/2}}{N^{2-}} \|\nabla Iu\|^3_{X^{0,\frac{1}{2}+}[0,\delta]} \\
 &\quad  +(\frac{1}{N^{4-}} + \frac{\delta}{N^{2-}})
 \|\nabla Iu\|^6_{X^{0,\frac{1}{2}+}[0,\delta]}
 + \frac{\delta}{N^{2-}} \|\nabla Iu\|^5_{X^{0,\frac{1}{2}+}[0,\delta]} \,.
\end{aligned}
\end{equation}
If we use \eqref{29} and \eqref{29''} we easily see that the decisive term is
$$
\frac{\delta^{1-}}{N^{1-}} \|\nabla Iu\|^4_{X^{0,\frac{1}{2}+}[0,\delta]}
+ \frac{\delta}{N^{2-}} \|\nabla Iu\|^6_{X^{0,\frac{1}{2}+}[0,\delta]}
\lesssim (\frac{N^{4(1-s+\epsilon)}}{N^{1-}}
+ \frac{N^{6(1-s+\epsilon)}}{N^{2-}}) \delta^{1-} \,.
$$
This is the bound for the increment of the modified energy from time $0$
to time $\delta$. Similarly we obtain the same bound for the increment
from time $t=k\delta$ to time $t=(k+1)\delta$ for $0 \le k \le T/{\delta}$,
 $ k \in{\mathbb N}$, provided we have a uniform bound
\begin{equation}\label{30}
\|\nabla Iu(k\delta)\|_{L^2}^2 \le 2 c_0  N^{2(1-s+\epsilon)} \,,
\end{equation}
which implies
\begin{equation}\label{30*}
\|\nabla Iu\|^2_{X^{0,\frac{1}{2}+}[k\delta,(k+1)\delta]}
\le 2 c_0 M N^{2(1-s+\epsilon)}
\end{equation}
by the local existence theorem. The number of iteration steps to reach the given
 time $T$ is $T/{\delta}$, so that the increment of the energy from time $t=0$
to time $t=(k+1)\delta$, $0 \le k \le \frac{T}{\delta}$, is bounded by
$$
\frac{T}{\delta} (\frac{N^{4(1-s+\epsilon)}}{N^{1-}}
+ \frac{N^{6(1-s+\epsilon)}}{N^{2-}}) \delta^{1-} \le c_0 N^{2(1-s)}
$$
independent of $k$, if $T N^{1-} N^{4(1-s+\epsilon)} \ll N^{2(1-s)} $ and
 $T N^{-2+} N^{6(1-s+\epsilon)} \ll N^{2(1-s)}$.
These conditions are fulfilled for $N$ sufficiently large, if
\begin{equation}
\label{epsilon}s > \frac{1}{2}+2\epsilon \Longleftrightarrow
\epsilon < \frac{s-\frac{1}{2}}{2} \,,
\end{equation}
 as one easily calculates. Choosing $\epsilon$ sufficiently small this
condition is fulfilled under our assumption $ s > 1/2 $. We recall again
that we used \eqref{30}. We arrive at
\begin{equation} \label{30'}
|E(Iu(t))| \le 2c_0 N^{2(1-s)} \quad \forall t \in[0,(k+1)\delta] \,, \;
 0\le k \le \frac{T}{\delta} ,  \; k \text{ fixed}.
\end{equation}
Now we consider the cases where either (A1) and (A2) or else (A1) and (A3)
hold separately.

If (A1) and (A2) hold we have  $\widehat{W}(\xi) > 0$, which immediately 
implies that the energy functional is positive definite, 
both terms in \eqref{E'} are namely nonnegative, so that one gets
$$ 
\|\nabla Iu(t)\|_{L^2}^2 \le E(I(u(t)) \le 2c_0 N^{2(1-s)} 
$$
for $ 0\le t \le (k+1)\delta$ and $0 \le k < \frac{T}{\delta}$, 
where $c_0$ is independent of $k$ and where we can choose $\epsilon =0$. 
Remark that on the right-hand side the same constant $2c_0$ appears
as in \eqref{30}.
Thus step by step after $\sim \frac{T}{\delta}$ steps we obtain the desired 
a-priori bound
$$ 
\|\nabla Iu(t)\|_{L^2} \le c(T) \quad \forall 0 \le t \le T \,. 
$$ 
Thus we are done in this case (modulo the results of the next section).

If (A1) and (A3) hold, the energy functional is not necessarily positive 
definite and it is more difficult to obtain a bound for
$\|\nabla Iu(t)\|_{L^2}$ from energy bounds.

We follow the computations of de Laire \cite{L} in this case and obtain
\begin{equation} \label{31}
\begin{aligned}
E(Iu) & = \|\nabla Iu\|^2_{L^2}
+ \frac{1}{2} \int (W*(|Iu|^2 + 2 \operatorname{Re} Iu))(|Iu|^2
+ 2 \operatorname{Re} Iu) dx \\
& =\|\nabla Iu\|^2_{L^2} + \tilde{I}_1(Iu) + \tilde{I}_2(Iu) + \tilde{I}_3(Iu) \,,
 \end{aligned}
\end{equation}
 where
 \begin{gather*}
 \tilde{I}_1(Iu)  := \int (W* \operatorname{Re} Iu) \operatorname{Re} Iu dx \\
 \tilde{I}_2(Iu)  := \frac{1}{2} \int (W* |Iu|^2)(|(Iu)_1|^2 + 4 \operatorname{Re}(Iu)_1) dx \\
\tilde{I}_3(Iu)  := \frac{1}{2} \int (W* |Iu|^2)(|(Iu)_2|^2 + 4 \operatorname{Re}(Iu)_2) dx \,.
\end{gather*}
Here
$$
(Iu)_1 := Iu \chi_{\{|Iu| \le 5\}} \,, \quad
(Iu)_2 := Iu \chi_{\{|Iu| > 5\}} \,.
$$
We used that $W$ is even which implies
$$
\int (W*|Iu|^2) \operatorname{Re} Iu \,dx
= \int (W* \operatorname{Re} Iu) |Iu|^2 dx \,.
$$
Using $W \in L^1(\mathbb{R}^3)$ we easily see that
$$
|\tilde{I}_1(Iu)| +  |\tilde{I}_2(Iu)| \lesssim \|Iu\|^2_{L^2} \,.
$$
Moreover using (A3) we obtain
\begin{align}\label{31'}
\tilde{I}_3(Iu)
& \ge \frac{1}{2} \int(W*|Iu|^2)(|(Iu)_2|^2 - 4 |(Iu)_2|) dx \\
& = \frac{1}{2} \int(W*|Iu|^2)|(Iu)_2|(|(Iu)_2| - 4) dx \\
& \ge \frac{1}{2} \int(W*|Iu|^2)|(Iu)_2| dx =: J_3(Iu) \ge 0 \,.
\end{align}
This implies by \eqref{31}
\begin{equation} \label{31''}
\|\nabla Iu\|^2_{L^2} + \tilde{I}_3(Iu) \le |E(Iu)| + a \|Iu\|^2_{L^2} \,.
\end{equation}
In order to estimate $\|Iu\|^2_{L^2}$ we apply $I$ to the differential
equation \eqref{4}, multiply with $iIu$ and take the real part. This leads to
\begin{equation} \label{32}
\frac{d}{dt} \|Iu\|^2_{L^2} = Im \langle F(Iu)-IF(u),Iu \rangle
- Im \langle F(Iu),Iu \rangle \,.
\end{equation}
Let us consider the first term on the right-hand side. We obtain
\begin{align*}
 Im \langle F(Iu)-IF(u),Iu \rangle
&= Im \langle (W*|Iu|^2)Iu-I((W*|u|^2)u),Iu \rangle \\
&\quad + 2 Im \langle(W*(\operatorname{Re}Iu))Iu
 -I((W*\operatorname{Re} u)u),Iu \rangle \\
&\quad + Im \langle (W*|Iu|^2)-I(W*|u|^2),Iu \rangle \\
&\quad + 2 Im \langle W*(\operatorname{Re} Iu)-I(W*\operatorname{Re} u),Iu \rangle \\
&= Im \langle I(W*|u|^2)-(W*|Iu|^2),Iu \rangle \,.
\end{align*}
Now we claim
\begin{equation} \label{33}
\int_{k\delta}^{(k+1)\delta} \langle I(W*|u|^2)-(W*|Iu|^2),Iu \rangle dt
 \lesssim N^{-2} \delta \|\nabla Iu\|^3_{X^{0,\frac{1}{2}+}[k\delta,(k+1)\delta]} \,.
\end{equation}
Using $|\tilde{W}(\xi)| \lesssim \frac{1}{\langle \xi \rangle^2} $ and defining
$$
M(\xi_1,\xi_2) := \frac{|m(\xi_1+\xi_2)-m(\xi_1)m(\xi_2)|}{m(\xi_1)m(\xi_2)}
\cdot \frac{1}{\langle \xi_1 + \xi_2 \rangle^2 |\xi_1||\xi_2||\xi_3|}
$$
we have to show
$$
 A:= \int_{k\delta}^{(k+1)\delta} \int_* M(\xi_1,\xi_2)
\prod_{i=1}^3 \widehat{u}_i(\xi_i,t) \,d\xi_1 \,d\xi_2 \,d\xi_3 dt
\lesssim N^{-2} \delta \prod_{i=1}^3 \|u_i\|^3_{X^{0,\frac{1}{2}+}} \,,
 $$
where * denotes integration over $\{\sum_{i=1}^3 \xi_i = 0 \}$.
 We assume without loss of generality $|\xi_1| \ge |\xi_2|$.

\noindent \textbf{Case 1:} $|\xi_1| \ge |\xi_2| \gtrsim N$: 
We have
$$
 M(\xi_1,\xi_2) \lesssim (\frac{|\xi_1|}{N})^{1/2}
(\frac{|\xi_2|}{N})^{1/2} \frac{1}{\langle \xi_1
+ \xi_2 \rangle^2 |\xi_1||\xi_2|} \,. 
$$
Thus
\begin{align*}
A & \lesssim N^{-2} \|u_1\|_{L^2_{x,t}} \|u_2\|_{L^2_{x,t}} \| 
 \langle D \rangle^{-2} D^{-1} u_3\|_{L^{\infty}_{x,t}} \\
& \lesssim N^{-2} \delta \|u_1\|_{L^{\infty}_t L^2_x}  
 \|u_2\|_{L^{\infty}_t L^2_x} \|D^{-1} u_3\|_{L^{\infty}_t L_x^6} \\
& \lesssim N^{-2} \delta \prod_{i=1}^3 \|u_i\|_{X^{0,\frac{1}{2}+}} \,.
\end{align*}

\noindent\textbf{Case 2:} $|\xi_1| \gtrsim N \gg |\xi_2|$ 
($ \Longrightarrow |\xi_3| \sim |\xi_1| \gtrsim N$).
By the mean value theorem we obtain
$$
M(\xi_1,\xi_2) \lesssim \frac{|\xi_2|}{|\xi_1|} \frac{1}{\langle \xi_1 
+ \xi_2 \rangle^2 |\xi_1||\xi_2||\xi_3|} \,. 
$$
Thus
\begin{align*}
A  \lesssim N^{-3} \| \langle D \rangle^{-2} (u_1 u_2)\|_{L^1_t L^2_x} 
\|u_3\|_{L^{\infty}_t L^2_x} 
&\lesssim N^{-3} \delta \|u_1 u_2\|_{L^{\infty}_t L^1_x} \|u_3\|_{L^{\infty}_t L^2_x} \\
& \lesssim N^{-3} \delta \prod_{i=1}^3 \|u_i\|_{X^{0,\frac{1}{2}+}} \,,
\end{align*}
which completes the proof of \eqref{33}.

Next we estimate the last term in \eqref{32}. We have
\begin{align*}
&|Im \langle F(Iu), Iu \rangle|  = |Im \int(W*(|Iu|^2 
+ 2 \operatorname{Re} Iu))(1+Iu)I\bar{u} dx | \\
& = |Im \int(W*(|Iu|^2 + 2 \operatorname{Re} Iu))I\bar{u} dx | \\
& \le 2 \int |W* \operatorname{Re} Iu||Iu| dx + \int(W*|Iu|^2)|(Iu)_1| dx 
 + \int (W*|Iu|^2)|(Iu)_2| dx  \\
& \lesssim \|Iu\|^2_{L^2}  + J_3(Iu) \lesssim  \|Iu\|^2_{L^2}  
+ \tilde{I}_3(Iu) \lesssim \|Iu\|_{L^2}^2 + |E(Iu)|
\end{align*}
by \eqref{31'} and \eqref{31''}.

 From \eqref{32} we conclude for $ k\delta \le t \le (k+1)\delta$:
\begin{align*}
& \|Iu(t)\|_{L^2}^2 - \|Iu(k\delta)\|_{L^2}^2 \\
& \lesssim c_1 N^{-2} \delta \|\nabla Iu\|^3_{X^{0,\frac{1}{2}+}[k\delta,(k+1)\delta]} + \int_{k\delta}^{(k+1)\delta} |E(Iu(s))|ds + \int_{k\delta}^t \|Iu(s)\|_{L^2}^2 ds ) \,.
\end{align*}
Now we have
$$
c_1(N^{-2} \delta \|\nabla Iu\|^3_{X^{0,\frac{1}{2}+}[k\delta,(k+1)\delta]} 
\le c_1 (2c_0)^{3/2} N^{-2}\delta N^{3(1-s) + 3\epsilon} 
\le c_2 \delta N^{2(1-s)} \,,
 $$
provided \eqref{30} holds (and therefore \eqref{30*}) and $\epsilon$ is 
sufficiently small. Using the uniform energy bound \eqref{30'} we obtain 
for $t \in [k\delta,(k+1)\delta]$:
$$ 
\|Iu(t)\|_{L^2}^2 \le \|Iu(k\delta)\|_{L^2}^2 + c_2 \delta N^{2(1-s)} 
+ 2c_0 \delta N^{2(1-s)} + c_1 \int_{k\delta}^t \|Iu(s)\|^2_{L^2} ds \,.
 $$
Gronwall's lemma implies
$$
 \sup_{k\delta \le t \le (k+1)\delta} \|Iu(t)\|^2_{L^2} 
\le (\|Iu(k\delta)\|_{L^2}^2 + c_3 \delta N^{2(1-s)}) e^{c_1 \delta} 
$$
under our assumptions \eqref{30'}
$$   
|E(Iu(t))| \le 2c_0 N^{2(1-s)} \quad \text{on} \quad [0,(k+1)\delta] 
$$
and \eqref{30}
$$   
\|\nabla Iu(k\delta)\|_{L^2}^2 \le 2c_0 N^{2(1-s+\epsilon)} \,. 
$$
Here $c_1$ and $c_3$ are independent of $k$. Using the bound for
 $\|Iu(k\delta)\|_{L^2}^2$ this implies
\begin{align*}
 \sup_{k\delta \le t \le (k+1)\delta} \|Iu(t)\|^2_{L^2} 
&\le [(\|Iu((k-1)\delta)\|_{L^2}^2 + c_3 \delta N^{2(1-s)}) e^{c_1 \delta} 
 +c_3 \delta N^{2(1-s)}] e^{c_1 \delta} \\
& = \|Iu((k-1)\delta)\|_{L^2}^2 \, e^{2c_1 \delta}
 + c_3 \delta N^{2(1-s)}( e^{2c_1 \delta} +  e^{c_1 \delta}) \,.
 \end{align*}
Iterating this procedure after $k \le T/\delta$ steps we arrive at
\begin{align*}
\sup_{k\delta \le t \le (k+1)\delta} \|Iu(t)\|_{L^2}^2 
& \le \|Iu_0\|_{L^2}^2 \, e^{\frac{T}{\delta} c_1 \delta} 
  + c_3 \delta N^{2(1-s)} \sum_{l=0}^{\frac{T}{\delta}} (e^{c_1 \delta})^l \\
& \le \|u_0\|_{L^2}^2 \, e^{c_1 T} + c_4 N^{2(1-s)} e^{c_1 T} \\
& \le \frac{c_0}{a} N^{2(1-s+\epsilon)}
\end{align*}
choosing $N$ so large that $e^{c_1 T} \ll N^{\epsilon}$ with a small 
$\epsilon > 0$, which fulfills \eqref{epsilon}, and $N$ also so large, 
that $\|u_0\|_{L^2}^2 \ll N^{2(1-s)}$. We used
$$ 
\sum_{l=0}^{\frac{T}{\delta}} (e^{c_1 \delta})^l
= \frac{(e^{c_1 \delta})^{\frac{T}{\delta}}-1}{e^{c_1 \delta} -1} 
\lesssim \frac{e^{c_1 T}}{\delta} \,. 
$$
This bound for $\|Iu(t)\|_{L^2}$  for $t \in [k\delta,(k+1)\delta]$ implies
 by \eqref{31''},\eqref{30},\eqref{30'}:
\begin{align*}
\|\nabla Iu(t)\|_{L^2}^2 
& \le |E(Iu(t))| + a \|Iu(t)\|_{L^2}^2 \\
& \le 2c_0 N^{2(1-s)} + c_0 N^{2(1-s+\epsilon)} \le 2c_0 N^{2(1-s+\epsilon)}
\end{align*}
for $t \in (k\delta,(k+1)\delta]$ (and choosing $N$ so large, that
 $N^{2\epsilon} \ge 2$), the same bound which we had for $t=k\delta$ (cf. \eqref{30}). 
By iteration we thus get
$$ 
\sup_{0 \le t \le T} \|\nabla Iu(t)\|_{L^2}^2 \le 2c_0 N^{2(1-s+\epsilon)} =: c(T) \,. 
$$
This completes the proof of the a-priori bound for $\|\nabla Iu(t)\|_{L^2}$ for the 
problem under the assumptions (A1) and (A3), so that now \eqref{16} holds in 
both cases. Thus the global well-posedness result is proven 
(modulo the results in the next section).

\section{Estimates for the modified energy}

In order to estimate the increment of the modified energy $E(Iu(t))$  
of a solution $u$ of the Cauchy problem \eqref{4},\eqref{5},\eqref{6} from 
time $t_0$ to time $t_0 + \delta$, say $t_0 = 0$ for ease of notation, 
we have to control its time derivative.
We calculate
\begin{align*}
\frac{d}{dt}E(Iu) &= 2 \operatorname{Re} \langle -\Delta Iu,Iu_t \rangle \\
&\quad + \frac{1}{2} \int (W*(Iu I \bar{u}_t + Iu_t I\bar{u} 
 + 2 \operatorname{Re} Iu_t))(|Iu|^2 + 2 \operatorname{Re} Iu) dx \\
& \quad + \frac{1}{2} \int (W*(|Iu|^2 + 2 \operatorname{Re} Iu))(Iu I \bar{u}_t 
 + I\bar{u} Iu_t + 2 \operatorname{Re} Iu_t) dx \\
& = 2 \operatorname{Re} \langle -\Delta Iu,Iu_t \rangle 
 + 2 \operatorname{Re} \langle (W*(|Iu|^2 
 + 2 \operatorname{Re} Iu))(1+Iu),Iu_t \rangle \,,
\end{align*}
where we used that $W$ is even, so that the second and third term coincide. Now
$$ 
-\Delta Iu = -iIu_t -IF(u) \,,
$$
so that
\begin{align*}
\frac{d}{dt}E(Iu)
& = 2 \operatorname{Re} \langle F(Iu)-IF(u),Iu_t \rangle \\
& = 2 \operatorname{Im} (\langle \nabla(F(Iu)-IF(u)),\nabla Iu \rangle 
 - \langle F(Iu)-IF(u),IF(u) \rangle)
\end{align*}
and
\begin{equation} \label{E}
|\frac{d}{dt}E(Iu)|
\le 2(| \langle \nabla(F(Iu)-IF(u)),\nabla Iu \rangle|
 + |\langle F(Iu)-IF(u),IF(u) \rangle|)
\end{equation}
with (cf. \eqref{6})
$$
F(u) = (1+u)(W*(|u|^2 +2\, Re\, u))\,.
$$
This especially shows the standard energy conservation law (setting $I = id$).

The estimates which now follow are given in terms of bounds of Fourier 
transforms of the corresponding functions. The only property of $W$ which 
we use is the bound $|\widehat{W}(\xi)| \lesssim \langle \xi \rangle^{-2}$, 
so that both cases, namely assuming either (A1) and (A2) or else (A1) and (A3) 
can be handled in the same way. The most critical cases are the terms of 
fourth and third order of the first term on the right-hand side of \eqref{E} 
and the term of sixth order of the second term. In fact we shall refer to 
the estimates in the case of the local Gross-Pitaevskii equation, 
where $\widehat{W}=1$, in our earlier paper \cite{Pe} for the remaining terms
 of lower order on the right-hand side of \eqref{E}.

We start with the terms of highest order in the first term. 
Taking again the time interval $[0,\delta]$ instead of $[k\delta,(k+1)\delta]$ 
just for the ease of notation we claim
\begin{equation} \label{50}
\begin{aligned}
&\big|\int_0^{\delta} \hspace{-0.6em}\langle \nabla((W*|Iu|^2)Iu
- I((W*|u|^2)u)),\nabla Iu \rangle dt \big| \\
&\lesssim (\frac{\delta^{1-}}{N^{1-}} + \frac{\delta^{1/2}}{N^{2-}})
 \| \nabla Iu \|^4_{X^{0,\frac{1}{2}+}[0,\delta]}  .
\end{aligned}
\end{equation}
Here and in the following we use dyadic decompositions with respect to the
space variables $\xi_i$, where $|\xi_i| \sim N_i$ with
$N_i =  2^{k_i}$, $k_i \in{\mathbb Z}$. In order to sum the dyadic parts
at the end we always need a convergence generating factor
$\frac{1\wedge N^{0+}_{min}}{N^{0+}_{max}}$, where $N_{min}$ and $N_{max}$
 is the smallest and the largest of the numbers $N_i$, respectively.
$N_{max} \ge N (\ge 1)$ can be assumed in all cases, because otherwise
our multiplier $M$ is identically zero. We have to take care of low frequencies
 especially, because we need an estimate in terms of $ \nabla Iu $.
Assuming without loss of generality that the Fourier transforms are nonnegative
we have to show:
$$
A:= \int_0^{\delta} \int_* M(\xi_1,\xi_2,\xi_3)
 \prod_{i=1}^4 \widehat{u_i}(\xi,t) \,d\xi_1\dots \,d\xi_4 dt
\lesssim (\frac{\delta^{1-}}{N^{1-}} + \frac{\delta^{1/2}}{N^{2-}})
\prod_{i=1}^4 \|u_i\|_{X^{0,\frac{1}{2}+}[0,\delta]},
$$
where * always denotes integration over $\{\sum \xi_i = 0\}$, and
$$
M(\xi_1,\xi_2,\xi_3)
:= \frac{|m(\xi_1 +\xi_2+\xi_3)-m(\xi_1)m(\xi_2)m(\xi_3)|}{m(\xi_1)m(\xi_2)m(\xi_3)}
\cdot \frac{|\xi_1 + \xi_2 + \xi_3|}{\langle \xi_1
+ \xi_2 \rangle^2 |\xi_1||\xi_2||\xi_3|} \,.
$$

\noindent\textbf{Case 1:} $N_3 \gg N_1 \ge N_2$.
In this case we have $N_3 \sim N_4 \gtrsim N$, $N_2 = N_{min}$ and 
$N_3 \sim N_{max}$. By the mean value theorem we obtain
$$
M(\xi_1,\xi_2,\xi_3) \lesssim \frac{N_1}{N_3} \cdot 
\frac{1}{N_1 |\xi_2| \langle \xi_1 + \xi_2 \rangle^2} 
$$
and by H\"older's inequality, Sobolev's embedding and Strichartz' estimate
 we obtain
\begin{align*}
A 
& \lesssim \frac{1}{N_3} \| \langle D 
 \rangle^{-2}(D^{-1} u_2 u_1)\|_{L^{\infty -}_t L^{\infty -}_x} 
 \|u_3\|_{L^{2+}_t L^{2+}_x} \|u_4\|_{L^2_t L^2_x} \\
& \lesssim \frac{\delta^{1-}}{N_3} \|D^{-1} u_2 u_1\|_{L^{\infty -}_t
  L^{\frac{3}{2}+}_x} \|u_3\|_{L^{\infty}_t L^{2+}_x} \|u_4\|_{L^{\infty}_t L^2_x} \\
& \lesssim \frac{\delta^{1-}}{N_3} \|D^{-1} u_2\|_{L^{\infty}_t L^{6+}_x}
  \|u_1\|_{L^{\infty -}_t L^{2+}_x} \|u_3\|_{L^{\infty}_t L^{2+}_x} 
 \|u_4\|_{L^{\infty}_t L^2_x} \\
& \lesssim \frac{\delta^{1-}(1 \wedge N_{\min}^{0+})}{N_{\max}^{0+} N^{1-}}
 \prod_{i=1}^4 \|u_i\|_{X^{0,\frac{1}{2}+}} \,,
\end{align*}
using $\| D^{-1}u_2\|_{L^{\infty}_t L^{6+}_x} \lesssim 
N_2^{0+} \|u_2\|_{L^{\infty}_t L^2_x}$ and $\|u_3\|_{L^{\infty}_t L^{2+}_x} 
 \lesssim N_3^{0+} \|u_3\|_{L^{\infty}_t L^2_x}$.

\noindent\textbf{Case 2:} $N_1 \gg N_2,N_3$ and $ N_1 \gtrsim N$. 
We have similarly as in case 1:
$$
 M(\xi_1,\xi_2,\xi_3) \lesssim \frac{N_3}{N_1} \cdot \frac{1}{|\xi_2| N_3 
\langle \xi_1 + \xi_2 \rangle^2} 
$$
and get the same estimate as in case 1 interchanging the roles 
of $N_1$ and $N_3$. 

\noindent\textbf{Case 3:} $N_1 \sim N_3$ ($\Longrightarrow N_1,N_3 \gtrsim N$).
In this case we obtain
\begin{align*}
 M(\xi_1,\xi_2,\xi_3)
&\lesssim (\frac{N_1}{N})^{\frac{1}{2}-} (\frac{N_3}{N})^{\frac{1}{2}-} 
 \langle\frac{N_1}{N}\rangle^{\frac{1}{2}-} \frac{1}{|\xi_2||\xi_3| 
 \langle \xi_1 + \xi_2 \rangle^2} \\
&\lesssim \frac{1}{N_3^{0+} N^{1-}} \langle \frac{N_2}{N}\rangle^{\frac{1}{2}-} 
\frac{1}{N_2 \langle \xi_1 + \xi_2 \rangle^2} \,.
 \end{align*}
\textbf{a.} $N_2 \gtrsim N$. We have
\begin{align*}
A & \lesssim \frac{1}{N_3^{0+} N^{2-}} \| \langle D \rangle^{-2}( u_1 u_2)\|_{L^2_t L^{3 -}_x} \|u_3\|_{L^{\infty}_t L^2_x} \|u_4\|_{L^2_t L^{6+}_x} \\
& \lesssim \frac{\delta^{1/2}N_4^{0+}}{N_3^{0+} N^{2-}} \|u_1 u_2\|_{L^{\infty}_t L^{1+}_x} \|u_3\|_{L^{\infty}_t L^2_x} \|u_4\|_{X^{0,\frac{1}{2}+}} \\
& \lesssim \frac{\delta^{1/2}(1 \wedge N_{\min}^{0+})}{N_{\max}^{0+} N^{2-}} \prod_{i=1}^4 \|u_i\|_{X^{0,\frac{1}{2}+}} \,,
\end{align*}
using $ \|u_4\|_{L^2_t L^{6+}_x} \lesssim N_4^{0+} \|u_4\|_{L_t^2 L_x^6} 
\lesssim \|u_4\|_{X^{0,\frac{1}{2}+}}$ by Strichartz' estimate. \\
\textbf{b.} $ N_2 \ll N$.
In this case we obtain
\begin{align*}
A & \lesssim \frac{1}{N_{max}^{0+} N^{1-}} \| \langle D 
 \rangle^{-2}( u_1 D^{-1}u_2)\|_{L^1_t L^{\infty -}_x} \|u_3\|_{L^{\infty}_t L^2_x} 
 \|u_4\|_{L^{\infty}_t L^{2+}_x} \\
& \lesssim \frac{\delta}{N_{\max}^{0+} N^{1-}} \|u_1 D^{-1}u_2\|_{L^{\infty}_t 
 L^{\frac{3}{2}+}_x} \|u_3\|_{L^{\infty}_t L^2_x} \|u_4\|_{L^{\infty}_t L^{2+}_x} \\
& \lesssim \frac{\delta}{N_{\max}^{0+} N^{1-}} \|u_1\|_{L^{\infty}_t L^2_x} 
 \|D^{-1}u_2\|_{L^{\infty}_t L^{6+}_x} \|u_3\|_{L^{\infty}_t L^2_x} 
 \|u_4\|_{L^{\infty}_t L^{2+}_x} \\
& \lesssim \frac{\delta(1 \wedge N_{\min}^{0+})}{N_{\max}^{0+} N^{1-}} 
 \prod_{i=1}^4 \|u_i\|_{X^{0,\frac{1}{2}+}} \,.
\end{align*}

\noindent\textbf{Case 4:} $N_1 \sim N_2 \gtrsim N_3$ 
($\Longrightarrow N_1,N_2 \gtrsim N$).
In this case we obtain
\begin{align*}
 M(\xi_1,\xi_2,\xi_3)
& \lesssim (\frac{N_1}{N})^{\frac{1}{2}-} (\frac{N_2}{N})^{\frac{1}{2}-} 
 \langle\frac{N_3}{N}\rangle^{\frac{1}{2}-} \frac{1}{|\xi_1||\xi_3| 
 \langle \xi_1 + \xi_2 \rangle^2} \\
 &\lesssim \frac{1}{N_{\max}^{0+} N^{1-}} \langle \frac{N_3}{N}
 \rangle^{\frac{1}{2}-} \frac{1}{N_3 \langle \xi_1 + \xi_2 \rangle^2} \,.
\end{align*}
 The case $N_3 \gtrsim N$ is handled like case 3a, whereas in the case 
$N_3 \ll N$ we obtain
\begin{align*}
A & \lesssim \frac{1}{N_{\max}^{0+} N^{1-}} \| \langle D \rangle^{-2}
 ( u_1 u_2)\|_{L^1_t L^{3 -}_x} \|D^{-1}u_3\|_{L^{\infty}_t L^{6+}_x} 
 \|u_4\|_{L^{\infty}_t L^{2+}_x} \\
& \lesssim \frac{\delta (1 \wedge N_{\min}^{0+})}{N_{\max}^{0+} N^{1-}} 
 \|u_1 u_2\|_{L^{\infty}_t L^{1}_x} \|u_3\|_{L^{\infty}_t L^2_x} 
 \|u_4\|_{L^{\infty}_t L^2_x} \\
& \lesssim \frac{\delta (1 \wedge N_{\min}^{0+})}{N_{\max}^{0+} N^{1-}}
 \prod_{i=1}^4 \|u_i\|_{X^{0,\frac{1}{2}+}} \,.
\end{align*}

\noindent\textbf{Case 5:} $N_2 \sim N_3$ ($\Longrightarrow N_1 \gtrsim N_2 \sim N_3$).
If $N_1 \gg N_2$ we are in the situation of case 2, otherwise 
$N_1 \sim N_2 \sim N_3 \gtrsim N$, so that
\begin{align*}
 M(\xi_1,\xi_2,\xi_3)
& \lesssim (\frac{N_1}{N})^{\frac{1}{2}-} (\frac{N_2}{N})^{\frac{1}{2}-} 
 (\frac{N_3}{N})^{\frac{1}{2}-} \frac{1}{N_2 N_3 \langle \xi_1 + \xi_2 \rangle^2} \\
 &\lesssim \frac{1}{N_{\max}^{0+} N^{2-}} \frac{1}{N_2 \langle \xi_1 
+ \xi_2 \rangle^2}
 \end{align*}
 as in case 3a.

 Dyadic summation gives estimate \eqref{50}.

 We next consider the cubic part of the first term on the right-hand side 
of \eqref{E}.  We claim
\begin{equation}  \label{51}
 \big|\int_0^{\delta} \langle \nabla(W*(|Iu|^2)-I(W*|u|^2)),\nabla Iu \rangle dt\big|
 \lesssim \frac{\delta^{1/2}}{N^{2-}}
 \|\nabla Iu\|^3_{X^{0,\frac{1}{2}+}[0,\delta]} \,.
 \end{equation}
We have to show
$$ 
B:= \int_0^{\delta} \int_* M(\xi_1,\xi_2) \prod_{i=1}^3
 \widehat{u_i}(\xi_i,t) \,d\xi_1 \,d\xi_2 \,d\xi_3 dt 
\lesssim \frac{\delta^{1/2}}{N^{2-}}
\prod_{i=1}^3 \|u_i\|_{X^{0,\frac{1}{2}+}[0,\delta]}  
$$
with
$$
M(\xi_1,\xi_2) := \frac{|m(\xi_1+\xi_2)-m(\xi_1)m(\xi_2)|}{m(\xi_1)m(\xi_2)} 
\cdot \frac{|\xi_1 + \xi_2 |}{\langle \xi_1 + \xi_2 \rangle^2 |\xi_1||\xi_2|} \,. 
$$

\noindent\textbf{Case 1:} $ N_1 \ge N_2 \gtrsim N$.
We have
$$ 
M(\xi_1,\xi_2) \lesssim (\frac{N_1}{N})^{1/2} (\frac{N_2}{N})^{1/2}
\frac{1}{N_1 N_2 \langle \xi_1 + \xi_2 \rangle} \,, 
$$
so that by Strichartz' estimate:
\begin{align*}
B & \lesssim \frac{1}{N_{\max}^{0+} N^{2-}} \| \langle D 
 \rangle^{-1}( u_1 u_2)\|_{L^1_t L^{2 -}_x}  \|u_3\|_{L^{\infty}_t L^{2+}_x} \\
& \lesssim \frac{1}{N_{\max}^{0+} N^{2-}} \|u_1 u_2\|_{L^1_t L^{\frac{3}{2}+}_x} 
 \|u_3\|_{L^{\infty}_t L^{2+}_x}  \\
& \lesssim \frac{\delta^{1/2}}{N_{\max}^{0+} N^{2-}}
 \|u_1\|_{L^{\infty}_t L^2_x} \| u_2\|_{L^2_t L^{6+}_x} 
 \|u_3\|_{L^{\infty}_t L^{2+}_x}  \\
& \lesssim \frac{\delta^{1/2} (1 \wedge N_{\min}^{0+})}{N_{\max}^{0+} N^{2-}}
\prod_{i=1}^3 \|u_i\|_{X^{0,\frac{1}{2}+}} \,.
\end{align*}

\noindent\textbf{Case 2:} $ N_1 \ge N \gg N_2$.
Similarly, by the mean value theorem we obtain
$$
 M(\xi_1,\xi_2) \lesssim \frac{N_2}{N_1}  \frac{1}{N_1 N_2 \langle \xi_1 
+ \xi_2 \rangle} \lesssim \frac{1}{N_{\max}^{0+} N^{2-} \langle \xi_1 
+ \xi_2 \rangle} 
$$
leading to the same bound as in case 1, thus \eqref{51} is proven.

Concerning the second cubic term we claim
\begin{equation}  \label{52}
 |\int_0^{\delta} \langle \nabla((W*Iu)Iu-I((W*u)u)),\nabla Iu \rangle dt| 
\lesssim \frac{\delta}{N^{1-}}
 \|\nabla Iu\|^3_{X^{0,\frac{1}{2}+}[0,\delta]} \,.
 \end{equation}
We again have to consider a term like $B$ but with
$$
M(\xi_1,\xi_2) := \frac{|m(\xi_1+\xi_2)-m(\xi_1)m(\xi_2)|}{m(\xi_1)m(\xi_2)} 
 \cdot \frac{|\xi_1 + \xi_2 |}{\langle \xi_1 \rangle^2 |\xi_1||\xi_2|} \,. 
$$
We concentrate on the more difficult case $N_2 \ge N_1$ and have to consider

\noindent\textbf{Case 1:} $ N_2 \sim N_1 \gtrsim N$.
We have
$$ 
M(\xi_1,\xi_2) \lesssim (\frac{N_1}{N})^{\frac{1}{2}-} 
(\frac{N_2}{N})^{\frac{1}{2}-} \frac{1}{N_1^3}
 \lesssim \frac{1}{N_{\max}^{0+} N_1^{3-}} \,.
 $$
Thus
\begin{align*}
B & \lesssim \frac{\delta}{N_{\max}^{0+} N_1^{3-}} 
 \| u_1\|_{L^{\infty}_t L^{\infty -}_x} \| u_2\|_{L^{\infty}_t L^2_x} 
  \|u_3\|_{L^{\infty}_t L^{2+}_x} \\
& \lesssim \frac{\delta (1 \wedge N_{\min}^{0+})}{N_{\max}^{0+} N_1^{3-}} N_1^{3/2}
  \prod_{i=1}^3  \|u_i\|_{L^{\infty}_t L^2_x}  \\
& \lesssim \frac{\delta (1 \wedge N_{\min}^{0+})}{N_{\max}^{0+} N^{\frac{3}{2}-}} 
 \prod_{i=1}^3 \|u_i\|_{X^{0,\frac{1}{2}+}} \,.
\end{align*}

\noindent\textbf{Case 2:} $ N_2 \gg N_1 $ ($\Longrightarrow N_2 \gtrsim N$).
By the mean value theorem we obtain
$$ 
M(\xi_1,\xi_2) \lesssim \frac{N_1}{N_2} \frac{1}{ \langle \xi_1 \rangle^2 N_1} 
= \frac{1}{N_2 \langle \xi_1 \rangle^2} \,, 
$$
so that
\begin{align*}
B & \lesssim \frac{\delta}{N_{\max}^{0+} N^{1-}} \|\langle D 
\rangle^{-2} u_1\|_{L^{\infty}_t L^{\infty}_x} \| u_2\|_{L^{\infty}_t L^2_x} 
 \|u_3\|_{L^{\infty}_t L^2_x} \\
& \lesssim \frac{\delta (1 \wedge N_{\min}^{0+})}{N_{\max}^{0+} N^{1-}} 
\prod_{i=1}^3 \|u_i\|_{L^{\infty}_t L^2_x} \,.
\end{align*}
Thus \eqref{52} follows.

We now have to consider the sixth order term on the right-hand side of \eqref{E}.
 Our aim is to show the following estimate:
\begin{equation} \label{53}
\begin{split}
& \big|\int_0^{\delta}\langle (W*|Iu|^2)Iu-I((W*|u|^2)u), I((W*|u|^2)u) 
\rangle dt\big| \\
&\lesssim (\frac{\delta}{N^{2-}} + \frac{1}{N^{4-}})
 \|\nabla Iu\|^6_{X^{0,\frac{1}{2}+}}\,.
\end{split}
 \end{equation}
 We have to show
\begin{align*}
C&:= \int_0^{\delta} \int_* M(\xi_1,\dots ,\xi_6) 
\prod_{i=1}^6 \widehat{u_i}(\xi_i,t) \,d\xi_1 \dots   \,d\xi_6 dt \\
&\lesssim \big(\frac{\delta}{N^{2-}} + \frac{1}{N^{4-}}\big) 
\prod_{i=1}^6 \|u_i\|_{X^{0,\frac{1}{2}+}[0,\delta]}  
\end{align*}
with
\begin{align*}
&M(\xi_1,\dots .,\xi_6) \\
&:=  \frac{|m(\xi_1+\xi_2+\xi_3)-m(\xi_1)m(\xi_2)m(\xi_3)|}{m(\xi_1)m(\xi_2)m(\xi_3)
 \langle \xi_1 + \xi_2 \rangle^2 }
 \cdot \frac{m(\xi_4 + \xi_5 + \xi_6)}{m(\xi_4)m(\xi_5)m(\xi_6) \langle \xi_4 
 + \xi_5 \rangle^2}\cdot \prod_{i=1}^6  |\xi_i|^{-1} \,.
\end{align*}
We assume without loss of generality $N_1 \ge N_2$ and $N_4 \ge N_5$. 

\noindent\textbf{Case 1:} $ N \gg N_4 \ge N_5$ and $N \gg N_6$.\\
\textbf{a.} $N_1 \ge N_2 \gtrsim N \gtrsim N_3$.
In this case we obtain
\begin{align*}
C & \lesssim (\frac{N_1}{N})^{1/2} (\frac{N_2}{N})^{1/2} 
 \frac{1}{N_1 N_2} \| \langle D \rangle^{-2}( u_1 u_2)\|_{L^1_t L^{\frac{3}{2} -}_x} 
 \|D^{-1} u_3\|_{L^{\infty}_t L^{6+}_x} \\
& \quad \times  \| \langle D \rangle^{-2}(D^{-1} u_4 D^{-1} u_5)\|_{L^{\infty}_t 
 L^{\infty}_x} \|D^{-1}u_6\|_{L^{\infty}_t L^{6+}_x} \\
& \lesssim \frac{\delta}{N_{\max}^{0+} N^{2-}} \| u_1 u_2\|_{L^{\infty}_t L^1_x} 
 \|D^{-1} u_3\|_{L^{\infty}_t L^{6+}_x} \|D^{-1} u_4 D^{-1} 
 u_5\|_{L^{\infty}_t L^{3+}_x} \|D^{-1}u_6\|_{L^{\infty}_t L^{6+}_x} \\
& \lesssim \frac{\delta}{N_{\max}^{0+} N^{2-}} \|u_1\|_{L^{\infty}_t L^2_x} 
 \|u_2\|_{L^{\infty}_t L^2_x} \|D^{-1} u_3\|_{L^{\infty}_t L^{6+}_x} 
 \|D^{-1} u_4\|_{L^{\infty}_t L^6_x}  \|D^{-1} u_5\|_{L^{\infty}_t L^{6+}_x} \\
& \quad \times \|D^{-1} u_6\|_{L^{\infty}_t L^{6+}_x}\\
&  \lesssim \frac{\delta(1 \wedge N_{\min}^{0+})}{N_{\max}^{0+} N^{2-}} 
 \prod_{i=1}^6 \|u_i\|_{X^{0,\frac{1}{2}+}} \,.
\end{align*}
\textbf{b.} $N_1 , N_3 \gtrsim N \gg N_2$. The estimate
$$
M(\xi_1,\dots ,\xi_6) \lesssim (\frac{N_1}{N})^{1/2}(\frac{N_3}{N})^{1/2}
 \frac{1}{N_1^2 \langle \xi_4 + \xi_5 \rangle^2 
|\xi_1||\xi_2| N_3 |\xi_4||\xi_5||\xi_6|}
 $$
implies
\begin{align*}
C & \lesssim \frac{\delta}{N_{\max}^{0+} N^{3-}} \|D^{-1} u_1\|_{L^{\infty}_t L^6_x} 
 \|D^{-1} u_2\|_{L^{\infty}_t L^{6+}_x} \| u_3\|_{L^{\infty}_t L^2_x} \\
&\quad\times \|\langle D \rangle^{-2} (D^{-1} u_4 D^{-1} u_5)\|_{L^{\infty}_t L^{\infty -}_x} \|D^{-1}u_6\|_{L^{\infty}_t L^{6+}_x} \\
& \lesssim \frac{\delta}{N_{\max}^{0+} N^{3-}} \|u_1\|_{L^{\infty}_t L^2_x} \|D^{-1} u_2\|_{L^{\infty}_t L^{6+}_x} \|u_3\|_{L^{\infty}_t L^2_x} \|D^{-1} u_4 D^{-1} u_5\|_{L^{\infty}_t L^{3+}_x}\\
& \quad\times \|D^{-1} u_6\|_{L^{\infty}_t L^{6+}_x}\\
&  \lesssim \frac{\delta(1 \wedge N_{\min}^{0+})}{N_{\max}^{0+} N^{3-}} \prod_{i=1}^6 \|u_i\|_{X^{0,\frac{1}{2}+}} \,.
\end{align*}
\textbf{c.} $N_1 , N_2 , N_3 \gtrsim N$.\\
In this case we obtain
\begin{align*}
C & \lesssim (\frac{N_1}{N})^{1/2} (\frac{N_2}{N})^{1/2} (\frac{N_3}{N})^{1/2} \frac{1}{N_1 N_2 N_3} \| \langle D \rangle^{-2}( u_1 u_2)\|_{L^1_t L^{3 -}_x} \|u_3\|_{L^{\infty}_t L^2_x} \\
& \quad\times \| \langle D \rangle^{-2}(D^{-1} u_4 D^{-1} u_5)\|_{L^{\infty}_t
  L^{\infty}_x} \|D^{-1}u_6\|_{L^{\infty}_t L^{6+}_x} \\
& \lesssim \frac{\delta}{N_{\max}^{0+} N^{3-}} \| u_1 u_2\|_{L^{\infty}_t 
 L^1_x} \|u_3\|_{L^{\infty}_t L^2_x} \|D^{-1} u_4 D^{-1} u_5\|_{L^{\infty}_t L^{3+}_x} 
 \|D^{-1}u_6\|_{L^{\infty}_t L^{6+}_x} \\
&  \lesssim \frac{\delta(1 \wedge N_{\min}^{0+})}{N_{\max}^{0+} N^{3-}} 
 \prod_{i=1}^6 \|u_i\|_{X^{0,\frac{1}{2}+}} \,.
\end{align*}

\noindent\textbf{Case 2:} $ N \gg N_4 \ge N_5$ and $N_6 \gtrsim N$.\\
\textbf{a.} $N_3 \gtrsim N \gtrsim N_1 \ge N_2$. By the mean value theorem we have
\begin{align*}
C & \lesssim \frac{N_1}{N_3} \cdot \frac{\delta^{1-}}{N_1 N_3 N_6} 
 \| \langle D \rangle^{-2}( u_1 D^{-1} u_2)\|_{L^{\infty -}_t L^{\infty}_x}
  \|u_3\|_{L^{\infty}_t L^2_x} \\
& \quad\times  \| \langle D \rangle^{-2}(D^{-1} u_4 D^{-1} u_5)\|_{L^{\infty}_t 
 L^{\infty}_x} \|u_6\|_{L^{\infty}_t L^2_x} \\
& \lesssim \frac{\delta^{1-}}{N_{\max}^{0+} N^{3-}} \| u_1 D^{-1}u_2\|_{L^{\infty -}_t
  L^{\frac{3}{2}+}_x} \|u_3\|_{L^{\infty}_t L^2_x}
  \|D^{-1} u_4 D^{-1} u_5\|_{L^{\infty}_t L^{3+}_x} \|u_6\|_{L^{\infty}_t L^2_x} \\
& \lesssim \frac{\delta^{1-}}{N_{\max}^{0+} N^{3-}} \|u_1\|_{L^{\infty -}_t L^{2+}_x} 
 \|D^{-1}u_2\|_{L^{\infty}_t L^{6+}_x} \|u_3\|_{L^{\infty}_t L^2_x} 
 \|D^{-1} u_4\|_{L^{\infty}_t L^6_x} \\
& \quad\times  \|D^{-1} u_5\|_{L^{\infty}_t L^{6+}_x} \|u_6\|_{L^{\infty}_t L^2_x}\\
&  \lesssim \frac{\delta^{1-}(1 \wedge N_{\min}^{0+})}{N_{\max}^{0+} N^{3-}} 
 \prod_{i=1}^6 \|u_i\|_{X^{0,\frac{1}{2}+}} \,.
\end{align*}
\textbf{b.} $N_1 \gtrsim N \gg N_2 \ge N_3$ (or  $N_1 \gtrsim N \gg N_3 \ge N_2$ 
by exchanging $u_2$ and $u_3$).                               \\
Similarly as in a. we use the mean value theorem and obtain
\begin{align*}
C & \lesssim \frac{N_2}{N_1} \cdot \frac{\delta}{N_1^2 N_2}
  \|D^{-1} u_1\|_{L^{\infty}_t L^6_x} \|u_2\|_{L^{\infty}_t L^2_x} 
 \|D^{-1}u_3\|_{L^{\infty}_t L^{6+}_x} \\
& \quad\times  \| \langle D \rangle^{-2}(D^{-1} u_4 D^{-1} u_5)\|_{L^{\infty}_t
  L^{\infty -}_x} \|D^{-1} u_6\|_{L^{\infty}_t L^6_x} \\
& \lesssim \frac{\delta}{N_1^3} \| u_1\|_{L^{\infty}_t L^2_x} \|u_2\|_{L^{\infty}_t 
 L^2_x} \|D^{-1} u_3\|_{L^{\infty}_t L^{6+}_x} \|D^{-1} u_4 D^{-1} u_5\|_{L^{\infty}_t 
 L^{3+}_x} \|u_6\|_{L^{\infty}_t L^2_x} \\
&  \lesssim \frac{\delta (1 \wedge N_{\min}^{0+})}{N_{\max}^{0+} N^{3-}} 
 \prod_{i=1}^6 \|u_i\|_{X^{0,\frac{1}{2}+}} \,.
\end{align*}
\textbf{c.} $N_1 \ge N_2 \gtrsim N \gg N_3$. 
We obtain
\begin{align*}
C & \lesssim (\frac{N_1}{N})^{\frac{1}{2}-}  (\frac{N_2}{N})^{\frac{1}{2}-} 
 \frac{1}{N_1 N_2 N_6} \|\langle D \rangle^{-2} (u_1 u_2)\|_{L^2_t L^3_x} 
 \|D^{-1}u_3\|_{L^{\infty}_t L^{6+}_x} \\
& \quad\times \| \langle D \rangle^{-2}(D^{-1} u_4 D^{-1} u_5)\|_{L^2_t 
 L^{\infty -}_x} \|u_6\|_{L^{\infty}_t L^2_x} \\
& \lesssim \frac{\delta}{N_{\max}^{0+} N^{3-}} \| u_1 u_2\|_{L^{\infty}_t 
 L^1_x} \|D^{-1} u_3\|_{L^{\infty}_t L^{6+}_x} \|D^{-1} u_4 D^{-1} u_5\|_{L^{\infty}_t
  L^{3+}_x} \|u_6\|_{L^{\infty}_t L^2_x} \\
&  \lesssim \frac{\delta (1 \wedge N_{\min}^{0+})}{N_{\max}^{0+} N^{3-}} 
 \prod_{i=1}^6 \|u_i\|_{X^{0,\frac{1}{2}+}} \,.
\end{align*}
\textbf{d.} $N_1 \ge N_2 \ge N_3 \gtrsim N$ or $N_1 \ge N_3 \ge N_2 \gtrsim N$.
This case can be handled similarly as case c. with an additional factor 
$(\frac{N_3}{N})^{\frac{1}{2}-}$.
\\
\textbf{e.} $N_3 \gtrsim N_1 \gtrsim N \gtrsim N_2$ 
(or $N_1 \ge N_3 \gtrsim N \gtrsim N_2$ by exchanging the roles of $u_1$ and $u_3$). 
We obtain
\begin{align*}
C & \lesssim (\frac{N_3}{N})^{\frac{1}{2}-}  (\frac{N_1}{N})^{\frac{1}{2}-} 
 \frac{\delta}{N_1 N_3 N_6} \|\langle D \rangle^{-2} (u_1 D^{-1}u_2)\|_{L^{\infty}_t
 L^{\infty}_x} \|u_3\|_{L^{\infty}_t L^2_x} \\
& \quad\times \| \langle D \rangle^{-2}(D^{-1} u_4 D^{-1} u_5)\|_{L^{\infty}_t 
 L^{\infty}_x} \|u_6\|_{L^{\infty}_t L^2_x} \\
& \lesssim \frac{\delta}{N_{\max}^{0+} N^{3-}} \| u_1 D^{-1}u_2\|_{L^{\infty}_t 
 L^{\frac{3}{2}+}_x} \|u_3\|_{L^{\infty}_t L^2_x} \|D^{-1} u_4 D^{-1} 
 u_5\|_{L^{\infty}_t L^{3+}_x} \|u_6\|_{L^{\infty}_t L^2_x} \\
& \lesssim \frac{\delta}{N_{\max}^{0+} N^{3-}} \| u_1\|_{L^{\infty}_t L^2_x} 
 \| D^{-1}u_2\|_{L^{\infty}_t L^{6+}_x} \|u_3\|_{L^{\infty}_t L^2_x} 
 \|D^{-1} u_4\|_{L^{\infty}_t L^6_x} \| D^{-1} u_5\|_{L^{\infty}_t L^{6+}_x} \\
& \quad\times  \|u_6\|_{L^{\infty}_t L^2_x} \\
&  \lesssim \frac{\delta (1 \wedge N_{\min}^{0+})}{N_{\max}^{0+} N^{3-}} 
\prod_{i=1}^6 \|u_i\|_{X^{0,\frac{1}{2}+}} \,.
\end{align*}
\textbf{f.} $ N_3 \ge N_1 \ge N_2 \gtrsim N$ (or $N_1 \ge N_3 \ge N_2 \gtrsim N$).
This case can be treated as case a. with an additional factor 
$(\frac{N_2}{N})^{\frac{1}{2}-}$.

\noindent\textbf{Case 3:} $N_4 \ge N_5 \gtrsim N$.\\
\textbf{a.} $N_1,N_2,N_3 \lesssim N$ and $N_6 \le N$.
We obtain
$$ 
M(\xi_1,\dots ,\xi_6) \lesssim \frac{1}{\langle \xi_1 
+ \xi_2 \rangle^2 \langle \xi_4 + \xi_5 \rangle^2
  N_4 N_5 |\xi_1||\xi_2||\xi_3||\xi_6|} 
(\frac{N_4}{N})^{\frac{1}{2}-} (\frac{N_5}{N})^{\frac{1}{2}-} \,, 
$$
so that
\begin{align*}
C & \lesssim \frac{\delta}{N_{\max}^{0+} N^{2-}} \|\langle D 
 \rangle^{-2}(D^{-1} u_1 D^{-1}u_2)\|_{L^{\infty}_t L^{\infty -}_x} 
 \|D^{-1}u_3\|_{L^{\infty}_t L^{6+}_x} \\ 
& \quad\times \|\langle D \rangle^{-2}(u_4 u_5)\|_{L^{\infty}_t L^{3/2}_x}  
 \|D^{-1}u_6\|_{L^{\infty}_t L^{6+}_x} \\
& \lesssim \frac{\delta}{N_{\max}^{0+} N^{2-}} \|D^{-1} u_1 D^{-1}u_2\|_{L^{\infty}_t 
 L^{3+}_x} \|D^{-1}u_3\|_{L^{\infty}_t L^{6+}_x} \| u_4 u_5\|_{L^{\infty}_t L^1_x} 
  \|D^{-1}u_6\|_{L^{\infty}_t L^{6+}_x} \\
&  \lesssim \frac{\delta (1 \wedge N_{\min}^{0+})}{N_{\max}^{0+} N^{2-}} 
 \prod_{i=1}^6 \|u_i\|_{X^{0,\frac{1}{2}+}} \,.
\end{align*}
\textbf{b.}  $N_1,N_2,N_3 \lesssim N$ and $N_6 \ge N$.
We argue similarly as in case a with an additional factor 
$(\frac{N_6}{N})^{\frac{1}{2}-}$ and get
\begin{align*}
C & \lesssim \frac{\delta}{N_{\max}^{0+} N^{3-}} \|\langle D \rangle^{-2}(D^{-1} 
 u_1 D^{-1}u_2)\|_{L^{\infty}_t L^{\infty -}_x} \|D^{-1}u_3\|_{L^{\infty}_t 
 L^{6+}_x} \\
& \quad\times \|\langle D \rangle^{-2}(u_4 u_5)\|_{L^{\infty}_t L^3_x}  
 \|u_6\|_{L^{\infty}_t L^2_x} \\
&  \lesssim \frac{\delta (1 \wedge N_{\min}^{0+})}{N_{\max}^{0+} N^{3-}} 
\prod_{i=1}^6 \|u_i\|_{X^{0,\frac{1}{2}+}} \,.
\end{align*}
\textbf{c.} $N_3 \gtrsim N$ and $N_1,N_2,N_6 \lesssim N$.
Replacing $\|D^{-1} u_3\|_{L^{\infty}_t L^{6+}_x}$ by 
$\|D^{-1} u_3\|_{L^{\infty}_t L^6_x}$ we obtain the same result as in case a.
\\
\textbf{d.} $N_3 \gtrsim N$ , $N_1,N_2 \lesssim N$ and $N_6 \gtrsim N$.
The additional factor $(\frac{N_6}{N})^{1/2}$ leads to the bound
\begin{align*}
C & \lesssim \frac{\delta}{N_{\max}^{0+} N^{3-}}
 \|\langle D \rangle^{-2}(D^{-1} u_1 D^{-1}u_2)\|_{L^{\infty}_t L^{\infty}_x} \|D^{-1}u_3\|_{L^{\infty}_t L^6_x} \\
& \quad\times  \|\langle D \rangle^{-2}(u_4 u_5)\|_{L^{\infty}_t L^3_x}  \|u_6\|_{L^{\infty}_t L^2_x} \\
&  \lesssim \frac{\delta (1 \wedge N_{\min}^{0+})}{N_{\max}^{0+} N^{3-}} \prod_{i=1}^6 \|u_i\|_{X^{0,\frac{1}{2}+}} \,.
\end{align*}
\textbf{e.} $N_1 \ge N_2 \gtrsim N$ and $N_3,N_6 \lesssim N$.
We obtain
$$
 M(\xi_1,\dots ,\xi_6) \lesssim \frac{(\frac{N_1}{N})^{\frac{1}{2}-}
 (\frac{N_2}{N})^{\frac{1}{2}-} (\frac{N_4}{N})^{\frac{1}{2}-}
 (\frac{N_5}{N})^{\frac{1}{2}-}}{\langle \xi_1 + \xi_2 \rangle^2 \langle \xi_4 
+ \xi_5 \rangle^2 N_1 N_2 N_4 N_5 |\xi_3||\xi_6|}  \,, 
$$
so that
\begin{align*}
C & \lesssim \frac{\delta}{N_{\max}^{0+} N^{4-}} \|\langle D \rangle^{-2}(u_1 u_2)\|_{L^{\infty}_t L^{3 -}_x} \|D^{-1}u_3\|_{L^{\infty}_t L^{6+}_x} \\
& \quad\times  \|\langle D \rangle^{-2}(u_4 u_5)\|_{L^{\infty}_t L^3_x}  \|D^{-1}u_6\|_{L^{\infty}_t L^{6+}_x} \\
& \lesssim \frac{\delta}{N_{\max}^{0+} N^{4-}} \|u_1 u_2\|_{L^{\infty}_t L^1_x} \|u_3\|_{L^{\infty}_t L^2_x} \| u_4 u_5\|_{L^{\infty}_t L^1_x}      \|u_6\|_{L^{\infty}_t L^2_x} \\
&  \lesssim \frac{\delta (1 \wedge N_{\min}^{0+})}{N_{\max}^{0+} N^{4-}} \prod_{i=1}^6 \|u_i\|_{X^{0,\frac{1}{2}+}} \,.
\end{align*}
\textbf{f.} $N_1 \ge N_2 \gtrsim N$ , $N_3\lesssim N$ and $N_6 \gtrsim N$.
The additional factor $(\frac{N_6}{N})^{1/2}$ leads to
\begin{align*}
C & \lesssim \frac{\delta}{N_{\max}^{0+} N^{4-}} \|\langle D \rangle^{-2}(u_1 u_2)\|_{L^{\infty}_t L^{3 -}_x} \|D^{-1}u_3\|_{L^{\infty}_t L^{6+}_x} \\
& \quad\times  \|\langle D \rangle^{-2}(u_4 D^{-1}u_5)\|_{L^{\infty}_t L^{\infty -}_x}  \|u_6\|_{L^{\infty}_t L^2_x} \\
& \lesssim \frac{\delta}{N_{\max}^{0+} N^{4-}} \|u_1 u_2\|_{L^{\infty}_t L^1_x} \|u_3\|_{L^{\infty}_t L^2_x} \| u_4\|_{L^{\infty}_t L^2_x} \|D^{-1} u_5\|_{L^{\infty}_t L^6_x} \|u_6\|_{L^{\infty}_t L^2_x} \\
&  \lesssim \frac{\delta (1 \wedge N_{\min}^{0+})}{N_{\max}^{0+} N^{4-}} \prod_{i=1}^6 \|u_i\|_{X^{0,\frac{1}{2}+}} \,.
\end{align*}
\textbf{g.} $N_1 \ge N_2 \gtrsim N$ , $N_6\lesssim N$ and $N_3 \gtrsim N$.
This case can be treated as case f. with $u_3$ and $u_6$ exchanged.
\\
\textbf{h.} $N_1 \ge N_2 \gtrsim N$ and $N_3,N_6 \gtrsim N$ 
($\Longrightarrow N_{\min} \gtrsim N$).
We obtain
$$
 M(\xi_1,\dots ,\xi_6) \lesssim \prod_{i=1}^6 (\frac{N_i}{N})^{\frac{1}{2}-} 
\prod_{i=1}^6 N_i^{-1}  \frac{1}{\langle \xi_1 + \xi_2 \rangle^2 \langle \xi_4 
+ \xi_5 \rangle^2}  \,, 
$$
so that
\begin{align*}
C & \lesssim \frac{1}{N_{\max}^{0+} N^{6-}} \|\langle D \rangle^{-2}(u_1 u_2)\|_{L^{\infty}_t L^{3}_x} \|u_3\|_{L^2_t L^{6}_x} \\
& \quad\times  \|\langle D \rangle^{-2}(u_4 u_5)\|_{L^{\infty}_t L^3_x}  \|u_6\|_{L^2_t L^{6}_x} \\
& \lesssim \frac{1}{N_{\max}^{0+} N^{6-}} \|u_1 u_2\|_{L^{\infty}_t L^1_x} \|u_3\|_{X^{0,\frac{1}{2}+}} \| u_4 u_5\|_{L^{\infty}_t L^1_x}      \|u_6\|_{X^{0,\frac{1}{2}+}} \\
&  \lesssim \frac{1}{N_{\max}^{0+} N^{6-}} \prod_{i=1}^6 \|u_i\|_{X^{0,\frac{1}{2}+}} \,.
\end{align*}
\textbf{i.} $N_1 \gtrsim N \gg N_2$ and $N_3,N_6 \lesssim N$.
We obtain
$$ 
M(\xi_1,\dots ,\xi_6) \lesssim  \frac{(\frac{N_1}{N})^{\frac{1}{2}-}
(\frac{N_2}{N})^{\frac{1}{2}-}(\frac{N_4}{N})^{\frac{1}{2}-}
(\frac{N_5}{N})^{\frac{1}{2}-}}{N_1^2 \langle \xi_4
 + \xi_5 \rangle^2 N_1 N_2 |\xi_3| N_4 N_5 |\xi_6|}  \,, 
$$
so that
\begin{align*}
C & \lesssim \frac{1}{N_{\max}^{0+} N^{6-}} \|u_1\|_{L^2_t L^6_x}
 \|u_2\|_{L^2_t L^{6+}_x} \|D^{-1} u_3\|_{L^{\infty}_t L^{6+}_x} \\
& \quad\times \|\langle D \rangle^{-2}(u_4 u_5)\|_{L^{\infty}_t L^{3-}_x} 
  \|D^{-1} u_6\|_{L^{\infty}_t L^{6+}_x} \\
&  \lesssim \frac{1 \wedge N_{\min}^{0+}}{N_{\max}^{0+} N^{6-}} 
 \prod_{i=1}^6 \|u_i\|_{X^{0,\frac{1}{2}+}} \,.
\end{align*}
\textbf{j.} $N_1 \gtrsim N \gg N_2$ and $N_3,N_6 \gtrsim N$.
We obtain
$$
 M(\xi_1,\dots ,\xi_6) \lesssim  \frac{(\frac{N_1}{N})^{\frac{1}{2}-}
(\frac{N_3}{N})^{\frac{1}{2}-}(\frac{N_4}{N})^{\frac{1}{2}-}
(\frac{N_5}{N})^{\frac{1}{2}-}(\frac{N_6}{N})^{\frac{1}{2}-}}{N_1^2 \langle \xi_4 
+ \xi_5 \rangle^2 |\xi_1| |\xi_2|N_3 N_4 N_5 N_6}  \,, 
$$
thus using $N_1 \lesssim \max(N_3,N_4,N_5,N_6)$:
\begin{align*}
C & \lesssim \frac{1}{N_{\max}^{0+} N^{6-}} \|D^{-1}u_1\|_{L^{\infty}_t L^6_x}
 \|D^{-1}u_2\|_{L^2_t L^{6+}_x} \|u_3\|_{L^2_t L^{6}_x} \\
& \quad\times \|\langle D \rangle^{-2}(u_4 u_5)\|_{L^{\infty}_t L^{3-}_x}  
 \|u_6\|_{L^2_t L^{6+}_x} \\
&  \lesssim \frac{1 \wedge N_{\min}^{0+}}{N_{\max}^{0+} N^{6-}} 
\prod_{i=1}^6 \|u_i\|_{X^{0,\frac{1}{2}+}} \,.
\end{align*}
\textbf{k.} $N_1 \gtrsim N \gg N_2$ , $N_3 \lesssim N$ and $N_6 \lesssim N$.
This case can be handled like j. without the factor $(\frac{N_6}{N})^{\frac{1}{2}-}$ 
by exchanging $u_1$ and $u_6$.
\\
\textbf{l.} $N_1 \gtrsim N \gg N_2$ , $N_3 \lesssim N$ and $N_6 \gtrsim N$.
The mean value theorem gives the bound
\begin{align*}
C & \lesssim \frac{N_2}{N_1} \frac{ (\frac{N_4}{N})^{\frac{1}{2}-}(\frac{N_5}{N})^{\frac{1}{2}-}(\frac{N_6}{N})^{\frac{1}{2}-}}{N_1^2  N_2 N_4 N_5 N_6}         \|D^{-1}u_1\|_{L^{\infty}_t L^6_x}
 \|u_2\|_{L^2_t L^{6+}_x} \|D^{-1} u_3\|_{L^{\infty}_t L^{6+}_x} \\
& \quad\times  \|\langle D \rangle^{-2}(u_4 u_5)\|_{L^{\infty}_t L^{3-}_x} 
  \|u_6\|_{L^2_t L^{6}_x} \\
&  \lesssim \frac{1 \wedge N_{\min}^{0+}}{N_{\max}^{0+} N^{6-}} \prod_{i=1}^6 
 \|u_i\|_{X^{0,\frac{1}{2}+}} \,.
\end{align*}

\noindent\textbf{Case 4:} $N_4 \ge N \gg N_5,N_6$.\\
\textbf{a.} $N_3 \gtrsim N \gtrsim N_1 \ge N_2$.
The mean value theorem gives
\begin{align*}
C & \lesssim \frac{N_1}{N_3} \frac{\delta}{N_1  N_3 N_4}   
  \|\langle D \rangle^{-2} (u_1 D^{-1}u_2)\|_{L^{\infty}_t L^{6}_x} \|u_3\|_{L^{\infty}_t L^{2}_x} \\
& \quad\times  \|\langle D \rangle^{-2}(u_4 D^{-1}u_5)\|_{L^{\infty}_t L^{6-}_x}  \|D^{-1}u_6\|_{L^{\infty}_t L^{6+}_x} \\
& \lesssim \frac{\delta}{N_{\max}^{0+} N^{3-}}    
  \|u_1 D^{-1}u_2\|_{L^{\infty}_t L^{\frac{3}{2}+}_x} \|u_3\|_{L^{\infty}_t L^{2}_x} \\
& \quad\times  \|u_4 D^{-1}u_5\|_{L^{\infty}_t L^{\frac{3}{2}+}_x}  \|D^{-1}u_6\|_{L^{\infty}_t L^{6+}_x} \\
& \lesssim \frac{\delta}{N_{\max}^{0+} N^{3-}}   
 \|u_1\|_{L^{\infty}_t L^2_x} \|D^{-1}u_2\|_{L^{\infty}_t L^{6+}_x} \|u_3\|_{L^{\infty}_t L^{2}_x} \\
& \quad\times  \|u_4\|_{L^{\infty}_t L^2_x} \| D^{-1}u_5\|_{L^{\infty}_t L^{6+}_x}
   \|D^{-1}u_6\|_{L^{\infty}_t L^{6+}_x} \\
&  \lesssim \frac{\delta(1 \wedge N_{\min}^{0+})}{N_{\max}^{0+} N^{3-}}
  \prod_{i=1}^6 \|u_i\|_{X^{0,\frac{1}{2}+}} \,.
\end{align*}
\textbf{b.} $N_1 \ge N \gg N_2 \ge N_3$.
The mean value theorem implies
\begin{align*}
C & \lesssim \frac{N_2}{N_1} \cdot \frac{1}{N_1^2 N_1 N_2}    
  \|u_1\|_{L^2_t L^6_x} \|u_2\|_{L^2_t L^{6}_x} \|D^{-1}u_3\|_{L^{\infty}_t L^{6+}_x} \\
& \quad\times  \|\langle D \rangle^{-2}(D^{-1}u_4 D^{-1}u_5)\|_{L^{\infty}_t L^{3-}_x}
   \|D^{-1}u_6\|_{L^{\infty}_t L^{6+}_x} \\
&  \lesssim \frac{1 \wedge N_{\min}^{0+}}{N_{\max}^{0+} N^{4-}} 
 \prod_{i=1}^6 \|u_i\|_{X^{0,\frac{1}{2}+}} \,.
\end{align*}
\textbf{c.} $N_1 \ge N \gg N_3 \ge N_2$.
This case can be treated like case b. with $u_2$ and $u_3$ exchanged. 
\\
\textbf{d.} $N_3 \ge N_1 \gtrsim N \gg N_2$.
We obtain
\begin{align*}
C & \lesssim (\frac{N_1}{N})^{1/2}(\frac{N_3}{N})^{1/2} \frac{1}{N_1^2 N_1 N_3}         \|u_1\|_{L^2_t L^6_x} \|D^{-1}u_2\|_{L^{\infty}_t L^{6+}_x} \|u_3\|_{L^2_t L^{6}_x} \\
& \quad\times  \|\langle D \rangle^{-2}(D^{-1}u_4 D^{-1}u_5)\|_{L^{\infty}_t L^{3-}_x}  \|D^{-1}u_6\|_{L^{\infty}_t L^{6+}_x} \\
&  \lesssim \frac{1 \wedge N_{\min}^{0+}}{N_{\max}^{0+} N^{4-}} \prod_{i=1}^6 \|u_i\|_{X^{0,\frac{1}{2}+}} \,.
\end{align*}
\textbf{e.} $N_1 \ge N_2 \gtrsim N \gtrsim N_3$.
We obtain in this case
$$
 M(\xi_1,\dots ,\xi_6) \lesssim  \frac{(\frac{N_1}{N})^{\frac{1}{2}-}
(\frac{N_2}{N})^{\frac{1}{2}-}}{\langle \xi_1 + \xi_2 \rangle^2 N_4^2 N_1 N_2
 |\xi_3||\xi_4||\xi_5||\xi_6|}  \,,
$$
so that
\begin{align*}
C & \lesssim \frac{\delta}{N_{\max}^{0+} N^{4-}} \|\langle D 
 \rangle^{-2}(u_1 u_2)\|_{L^{\infty}_t L^{3-}_x} \|D^{-1}u_3\|_{L^{\infty}_t L^{6+}_x} \\
& \quad\times  \|D^{-1}u_4\|_{L^{\infty}_t L^6_x} \| D^{-1}u_5\|_{L^{\infty}_t L^{6+}_x}  \|D^{-1}u_6\|_{L^{\infty}_t L^{6+}_x} \\
& \lesssim  \frac{\delta(1 \wedge N_{\min}^{0+})}{N_{\max}^{0+} N^{4-}}   
  \|u_1 u_2\|_{L^{\infty}_t L^{1}_x} \prod_{i=3}^6 \|u_i\|_{L^{\infty}_t L^{2}_x}  \\
&  \lesssim \frac{\delta(1 \wedge N_{\min}^{0+})}{N_{\max}^{0+} N^{4-}}
  \prod_{i=1}^6 \|u_i\|_{X^{0,\frac{1}{2}+}} \,.
\end{align*}
\textbf{f.} $N_1 \ge N_3 \gtrsim N \gtrsim N_2$.
We obtain
\begin{align*}
C & \lesssim (\frac{N_1}{N})^{\frac{1}{2}-} (\frac{N_3}{N})^{\frac{1}{2}-}
      \frac{\delta}{N_4^2 N_1 N_3} \|\langle D \rangle^{-2}(u_1 D^{-1}u_2)\|_{L^{\infty}_t L^{\infty -}_x} \|u_3\|_{L^{\infty}_t L^{2}_x} \\
& \quad\times  \|D^{-1}u_4\|_{L^{\infty}_t L^6_x} \| D^{-1}u_5\|_{L^{\infty}_t L^{6+}_x}  \|D^{-1}u_6\|_{L^{\infty}_t L^{6+}_x} \\
 & \lesssim \frac{\delta}{N_{\max}^{0+} N^{4-}} \|u_1 D^{-1}u_2\|_{L^{\infty}_t
  L^{\frac{3}{2}+}_x} \|u_3\|_{L^{\infty}_t L^{2}_x} \\
& \quad\times  \|D^{-1}u_4\|_{L^{\infty}_t L^6_x} \| D^{-1}u_5\|_{L^{\infty}_t 
 L^{6+}_x}  \|D^{-1}u_6\|_{L^{\infty}_t L^{6+}_x} \\
& \lesssim  \frac{\delta}{N_{\max}^{0+} N^{4-}}    \|u_1\|_{L^{\infty}_t L^2_x} 
 \|D^{-1}u_2\|_{L^{\infty}_t L^{6+}_x}  \|u_3\|_{L^{\infty}_t L^{2}_x} \\
& \quad\times  \|D^{-1}u_4\|_{L^{\infty}_t L^6_x} \| D^{-1}u_5\|_{L^{\infty}_t 
 L^{6+}_x}  \|D^{-1}u_6\|_{L^{\infty}_t L^{6+}_x} \\
&  \lesssim \frac{\delta(1 \wedge N_{\min}^{0+})}{N_{\max}^{0+} N^{4-}} 
 \prod_{i=1}^6 \|u_i\|_{X^{0,\frac{1}{2}+}} \,.
\end{align*}
\textbf{g.} $N_1 \ge N_2 \gtrsim N$ and $N_3 \gtrsim N$.
This case can be treated as case f. with an additional factor 
$(\frac{N_2}{N})^{\frac{1}{2}-}$.

\noindent\textbf{Case 5:} $N_4,N_6 \ge N \gg N_5$.\\
\textbf{a.} $N_3 \gtrsim N \gtrsim N_1 \ge N_2$.
This case can be treated as case 2a, because the additional factor 
$(\frac{N_4}{N})^{\frac{1}{2}-}(\frac{N_6}{N})^{\frac{1}{2}-}$ is harmless,
 when one uses $N_4 \lesssim \max(N_3,N_6)$.
\\
\textbf{b.} $N_1 \gtrsim N \gg N_2 \ge N_3$ (or $N_1 \gtrsim N \gg N_3 \ge N_2$
 by exchanging $u_2$ and $u_3$).
We have by the mean value theorem
\begin{align*}
C & \lesssim \frac{N_2}{N_1} (\frac{N_4}{N})^{\frac{1}{2}-} 
 (\frac{N_6}{N})^{\frac{1}{2}-} \frac{1}{N_1^2 N_1 N_2 N_4 N_6} 
 \|u_1\|_{L^{\infty}_t L^{2+}_x} \|u_2\|_{L^2_t L^6_x} \|D^{-1}u_3\|_{L^{\infty}_t
  L^{6+}_x} \\
& \quad\times  \|\langle D \rangle^{-2} (u_4 D^{-1}u_5)\|_{L^{\infty}_t L^{\infty -}_x}  \|u_6\|_{L^2_t L^{6}_x} \\
 & \lesssim \frac{1}{N_{\max}^{0+} N^{6-}} \|u_1\|_{L^{\infty}_t L^{2}_x} 
 \|u_2\|_{X^{0,\frac{1}{2}+}} \|D^{-1}u_3\|_{L^{\infty}_t L^{6+}_x} \\
& \quad\times  \|u_4 D^{-1}u_5\|_{L^{\infty}_t L^{\frac{3}{2}+}_x} 
  \|u_6\|_{X^{0,\frac{1}{2}+}}   \\
& \lesssim  \frac{1}{N_{\max}^{0+} N^{6-}} \|u_1\|_{L^{\infty}_t L^{2}_x} 
 \|u_2\|_{X^{0,\frac{1}{2}+}}  \|D^{-1}u_3\|_{L^{\infty}_t L^{6+}_x} 
 \|u_4\|_{L^{\infty}_t L^2_x} \| D^{-1}u_5\|_{L^{\infty}_t L^{6+}_x}\\ 
&\quad \times \|u_6\|_{X^{0,\frac{1}{2}+}}   \\
&  \lesssim \frac{1 \wedge N_{\min}^{0+}}{N_{\max}^{0+} N^{6-}}
  \prod_{i=1}^6 \|u_i\|_{X^{0,\frac{1}{2}+}} \,.
\end{align*}
\textbf{c.} $N_1 \ge N_2 \gtrsim N \gg N_3$.
This case is treated like case 2c, because the additional factor
 $(\frac{N_4}{N})^{\frac{1}{2}-} (\frac{N_6}{N})^{\frac{1}{2}-}$
 can be handled using $N_4 \lesssim \max(N_1,N_6)$.
\\
\textbf{d.} $N_3 \ge N_1 \gtrsim N \gg N_2$ or  $N_1 \ge N_3 \gtrsim N \ge N_2$.
This case can be handled like case 2e, using $N_4 \lesssim \max(N_1,N_3,N_6)$.
\\
\textbf{e.} $N_1 \ge N_2 \gtrsim N$ and $N_3 \gtrsim N$.
This case is also treated like case 2e, because the additional factor
 $(\frac{N_2}{N})^{\frac{1}{2}-}(\frac{N_4}{N})^{\frac{1}{2}-}$  
 $ (\frac{N_6}{N})^{\frac{1}{2}-}$ is acceptable, using $N_2 \le N_1$ and 
$N_4 \lesssim \max(N_1,N_3,N_6)$.

\noindent\textbf{Case 6:} $N_4,N_5,N_6 \gtrsim N$.\\
\textbf{a.} $N_3 \gtrsim N \gg N_1,N_2$.
The mean value theorem allows to estimate
$$ 
M(\xi_1,\dots ,\xi_6) \lesssim \frac{N_1}{N_3} \cdot 
\frac{(\frac{N_4}{N})^{\frac{1}{2}-}(\frac{N_5}{N})^{\frac{1}{2}-} 
(\frac{N_6}{N})^{\frac{1}{2}-}}{\langle \xi_1 + \xi_2 \rangle^2 \langle \xi_4
 + \xi_5 \rangle^2 N_1 |\xi_2| |\xi_3| N_4 N_5 N_6}  
\lesssim \frac{1}{N_{\max}^{0+} N^{4-}} \,,
 $$
thus
\begin{align*}
C & \lesssim \frac{\delta}{N_{\max}^{0+} N^{4-}} \|\langle D \rangle^{-2}(u_1 D^{-1}u_2)\|_{L^{\infty}_t L^{\infty}_x} \|D^{-1}u_3\|_{L^{\infty}_t L^{6}_x} \\
& \quad\times  \|\langle D \rangle^{-2}(u_4 u_5)\|_{L^{\infty}_t L^{3}_x}  \|u_6\|_{L^{\infty}_t L^{2}_x} \\
 & \lesssim \frac{\delta}{N_{\max}^{0+} N^{4-}} \|u_1 D^{-1}u_2\|_{L^{\infty}_t L^{\frac{3}{2}+}_x} \|u_3\|_{L^{\infty}_t L^{2}_x}
 \|u_4 u_5\|_{L^{\infty}_t L^1_x}  \|u_6\|_{L^{\infty}_t L^{2}_x}  \\
& \lesssim  \frac{\delta}{N_{\max}^{0+} N^{4-}} \|u_1\|_{L^{\infty}_t L^{2}_x} \|D^{-1}u_2\|_{L^{\infty}_t L^{6+}_x} \|u_3\|_{L^{\infty}_t L^{2}_x} \|u_4\|_{L^{\infty}_t L^{2}_x}  \|u_5\|_{L^{\infty}_t L^{2}_x} \|u_6\|_{L^{\infty}_t L^{2}_x}  \\
&  \lesssim \frac{\delta(1 \wedge N_{\min}^{0+})}{N_{\max}^{0+} N^{4-}} \prod_{i=1}^6 \|u_i\|_{X^{0,\frac{1}{2}+}} \,.
\end{align*}
\textbf{b.} $N_1, N_2, N_3 \gtrsim N$ and without loss of generality $N_1 \ge N_2$
 and $N_4 \ge N_5$.
We have
\begin{align*}
C & \lesssim \prod_{i=1}^6 (\frac{N_i}{N})^{\frac{1}{2}-} 
\frac{\delta^{1-}}{N_1 N_3 N_4 N_6} 
\|\langle D \rangle^{-2}(u_1 D^{-1}u_2)\|_{L^{\infty -}_t L^{\infty}_x} 
\|u_3\|_{L^{\infty}_t L^{2}_x} \\
& \quad\times  \|\langle D \rangle^{-2}(u_4 D^{-1}u_5)\|_{L^{\infty -}_t
 L^{\infty}_x}  \|u_6\|_{L^{\infty}_t L^{2}_x} \\
 & \lesssim \frac{\delta^{1-}}{N_{\max}^{0+} N^{4-}} \|u_1 D^{-1}u_2\|_{L^{\infty -}_t
 L^{\frac{3}{2}+}_x} \|u_3\|_{L^{\infty}_t L^{2}_x}
 \|u_4 D^{-1}u_5\|_{L^{\infty -}_t L^{\frac{3}{2}+}_x}  \|u_6\|_{L^{\infty}_t L^{2}_x}  \\
& \lesssim  \frac{\delta^{1-}}{N_{\max}^{0+} N^{4-}} \|u_1\|_{L^{\infty -}_t L^{2+}_x} 
\|D^{-1}u_2\|_{L^{\infty}_t L^{6+}_x} \|u_3\|_{L^{\infty}_t L^{2}_x} 
 \|u_4\|_{L^{\infty -}_t L^{2+}_x}  \|D^{-1}u_5\|_{L^{\infty}_t L^{6}_x} \\
& \quad\times \|u_6\|_{L^{\infty}_t L^{2}_x}  \\
&  \lesssim \frac{\delta^{1-}(1 \wedge N_{\min}^{0+})}{N_{\max}^{0+} N^{4-}} 
\prod_{i=1}^6 \|u_i\|_{X^{0,\frac{1}{2}+}} \,.
\end{align*}
\textbf{c.} $N_1,N_3 \gtrsim N \gtrsim N_2$.
This case can be treated as case b. without the factor 
$(\frac{N_2}{N})^{\frac{1}{2}-}$. 
\\
\textbf{d.} $N_1 \ge N_2 \gtrsim N \gtrsim N_3$ and without loss of generality 
$N_4 \ge N_5$.
We obtain
\begin{align*}
C & \lesssim \prod_{i \neq 3} (\frac{N_i}{N})^{\frac{1}{2}-} 
 \frac{\delta}{N_1 N_2 N_4 N_6} \|\langle D \rangle^{-2}(u_1 u_2)\|_{L^{\infty}_t 
 L^{3}_x} \|D^{-1}u_3\|_{L^{\infty}_t L^{6+}_x} \\
& \quad\times  \|\langle D \rangle^{-2}(u_4 D^{-1}u_5)\|_{L^{\infty}_t L^6_x}  
 \|u_6\|_{L^{\infty}_t L^{2}_x} \\
 & \lesssim \frac{\delta(1 \wedge N_{\min}^{0+})}{N_{\max}^{0+} N^{4-}}
 \|u_1 u_2\|_{L^{\infty}_t L^1_x} \|u_3\|_{L^{\infty}_t L^{2}_x}
 \|u_4 D^{-1}u_5\|_{L^{\infty}_t L^{3/2}_x}  \|u_6\|_{L^{\infty}_t L^{2}_x}  \\
& \lesssim  \frac{\delta(1 \wedge N_{\min}^{0+})}{N_{\max}^{0+} N^{4-}} 
 \|u_1\|_{L^{\infty}_t L^{2}_x} \|u_2\|_{L^{\infty}_t L^{2}_x} 
 \|u_3\|_{L^{\infty}_t L^{2}_x} \|u_4\|_{L^{\infty}_t L^{2}_x}  
 \|D^{-1}u_5\|_{L^{\infty}_t L^{6}_x} \\
& \quad \times \|u_6\|_{L^{\infty}_t L^{2}_x}  \\
&  \lesssim \frac{\delta(1 \wedge N_{\min}^{0+})}{N_{\max}^{0+} N^{4-}} 
 \prod_{i=1}^6 \|u_i\|_{X^{0,\frac{1}{2}+}} \,.
\end{align*}
This completes the proof of \eqref{53}.

We now start to consider the fifth order terms and claim
\begin{equation} \label{54'}
 \big|\int_0^{\delta} \langle (W*|Iu|^2)Iu-I((W*|u|^2)u),
 I(W*|u|^2) \rangle dt\big| \lesssim \frac{\delta}{N^{2-}}
 \|\nabla Iu\|^5_{X^{0,\frac{1}{2}+}}  .
 \end{equation}
 We have to show
$$
 D:= \int_0^{\delta} \int_* M(\xi_1,\dots ,\xi_5) \prod_{i=1}^5 \widehat{u_i}(\xi_i,t)
 \,d\xi_1 \dots   \,d\xi_5 dt \lesssim \frac{\delta}{N^{2-}} 
\prod_{i=1}^5 \|u_i\|_{X^{0,\frac{1}{2}+}[0,\delta]} 
 $$
with
\begin{align*}
M(\xi_1,\dots .,\xi_5) 
&:=  \frac{|m(\xi_1+\xi_2+\xi_3)-m(\xi_1)m(\xi_2)m(\xi_3)|}{m(\xi_1)m(\xi_2)m(\xi_3) 
\langle \xi_1 + \xi_2 \rangle^2 }\\ 
&\quad\times \frac{m(\xi_4 + \xi_5)}{m(\xi_4)m(\xi_5) \langle \xi_4 + \xi_5 \rangle^2}
\cdot \prod_{i=1}^5  |\xi_i|^{-1} \,.
\end{align*}
We assume without loss of generality $N_1 \ge N_2$ and $N_4 \ge N_5$.

\noindent\textbf{Case 1:} $ N \gg N_4 \ge N_5$.\\
\textbf{a.} $N_1 \ge N_2 \gtrsim N \gtrsim N_3$.
In this case we obtain
\begin{equation}
\begin{aligned} \label{54}
D & \lesssim (\frac{N_1}{N})^{1/2} (\frac{N_2}{N})^{1/2}
 \frac{1}{N_1 N_2} \| \langle D \rangle^{-2}( u_1 u_2)\|_{L^1_t L^{\frac{6}{5}}_x}
 \|D^{-1} u_3\|_{L^{\infty}_t L^{6+}_x} \\
& \quad \times \| \langle D \rangle^{-2}(D^{-1} u_4 D^{-1} u_5)\|_{L^{\infty}_t
 L^{\infty-}_x}  \\
& \lesssim \frac{\delta}{N_{\max}^{0+} N^{2-}} \| u_1 u_2\|_{L^{\infty}_t L^1_x}
\|D^{-1} u_3\|_{L^{\infty}_t L^{6+}_x} \|D^{-1} u_4 D^{-1} u_5\|_{L^{\infty}_t
 L^{3+}_x}  \\
&  \lesssim \frac{\delta(1 \wedge N_{\min}^{0+})}{N_{\max}^{0+} N^{2-}}
 \prod_{i=1}^5 \|u_i\|_{X^{0,\frac{1}{2}+}} \,.
\end{aligned}
\end{equation}
\textbf{b.} $N_1 , N_3 \gtrsim N \gg N_2$.
We have
\begin{align*}
D & \lesssim (\frac{N_1}{N})^{1/2} (\frac{N_3}{N})^{1/2}
   \frac{\delta}{N_1 N_3} \| \langle D \rangle^{-2}( u_1 D^{-1}u_2)\|_{L^{\infty}_t
 L^2_x} \|u_3\|_{L^{\infty}_t L^2_x} \\
& \quad\times  \| \langle D \rangle^{-2}(D^{-1} u_4 D^{-1} u_5)\|_{L^{\infty}_t
 L^{\infty}_x}  \\
& \lesssim \frac{\delta}{N_{\max}^{0+} N^{2-}} \| u_1 D^{-1}u_2\|_{L^{\infty}_t
 L^{\frac{3}{2}+}_x} \|u_3\|_{L^{\infty}_t L^2_x} \|D^{-1} u_4 D^{-1} u_5\|_{L^{\infty}_t L^{3+}_x}  \\
&  \lesssim \frac{\delta(1 \wedge N_{\min}^{0+})}{N_{\max}^{0+} N^{2-}}
 \prod_{i=1}^5 \|u_i\|_{X^{0,\frac{1}{2}+}} \,.
\end{align*}
\textbf{c.} $N_1 ,N_2, N_3 \gtrsim N$.
This leads to
\begin{align*}
D & \lesssim (\frac{N_1}{N})^{1/2}(\frac{N_2}{N})^{1/2} (\frac{N_3}{N})^{1/2}
  \frac{\delta}{N_1 N_2  N_3} \| \langle D \rangle^{-2}( u_1 u_2)\|_{L^{\infty}_t
  L^2_x} \|u_3\|_{L^{\infty}_t L^2_x} \\
& \quad\times  \| \langle D \rangle^{-2}(D^{-1} u_4 D^{-1} u_5)\|_{L^{\infty}_t
 L^{\infty}_x}  \\
&  \lesssim \frac{\delta(1 \wedge N_{\min}^{0+})}{N_{\max}^{0+} N^{3-}}
 \prod_{i=1}^5 \|u_i\|_{X^{0,\frac{1}{2}+}} \,.
\end{align*}

\noindent\textbf{Case 2:} $N_4 \ge N_5 \gtrsim N$.
The additional factor $(\frac{N_4}{N})^{\frac{1}{2}-} 
(\frac{N_5}{N})^{\frac{1}{2}-}$ can be compensated by replacing the last factor
 in \eqref{54} by
$$
 N_4^{-1/2} N_5^{-1/2} \| D^{-1/2}u_4
D^{-1/2}u_5\|_{L^{\infty}_t L^{3/2}_x} \lesssim  N_4^{-1/2}
N_5^{-1/2} \|u_4\|_{L^{\infty}_t L^2_x} \|u_5\|_{L^{\infty}_t L^2_x} 
$$
leading to even an improved bound.

\noindent \textbf{Case 3:} $N_4 \ge N \ge N_5$.
We argue as before replacing the last factor in \eqref{54} by
$$ 
N_4^{-1/2} \| D^{-1/2}u_4 D^{-1}u_5\|_{L^{\infty}_t L^{2+}_x}
 \lesssim  N_4^{-1/2}  \|u_4\|_{L^{\infty}_t L^2_x} \|D^{-1}u_5\|_{L^{\infty}_t
 L^{6+}_x} \,.
$$
This proves \eqref{54'}.

Next we claim
\begin{equation}\label{54''}
 \big|\int_0^{\delta} \langle (W*|Iu|^2)Iu-I((W*|u|^2)u), I((W* Re\,u)u) 
\rangle dt\big| 
\lesssim \frac{\delta}{N^{2-}} \|\nabla Iu\|^5_{X^{0,\frac{1}{2}+}}  .
 \end{equation}
 We again consider $D$ with
\begin{align*}
&M(\xi_1,\dots .,\xi_5) \\
&:=  \frac{|m(\xi_1+\xi_2+\xi_3)-m(\xi_1)m(\xi_2)m(\xi_3)|}{m(\xi_1)m(\xi_2)m(\xi_3) 
\langle \xi_1 + \xi_2 \rangle^2 }
\cdot \frac{m(\xi_4 + \xi_5)}{m(\xi_4)m(\xi_5)  \langle \xi_4 \rangle^2}
\cdot \prod_{i=1}^5  |\xi_i|^{-1} \,.
\end{align*}
We argue similarly as in the previous case and consider only the more difficult 
case $N_5 \ge N_4$. 

\noindent\textbf{Case 1:} $N \gg N_5 \ge N_4$. 
The last term in \eqref{54} with a suitable change of the H\"older exponent 
in the first factor can be replaced by
\begin{align*}
\|\langle D \rangle^{-2} D^{-1} u_4 D^{-1} u_5\|_{L^{\infty}_t L^{6}_x} 
& \lesssim  \|\langle D \rangle^{-2} D^{-1}u_4\|_{L^{\infty}_t L^{\infty}_x} 
\|D^{-1} u_5\|_{L^{\infty}_t L^6_x} \\
&\lesssim  \| D^{-1}u_4\|_{L^{\infty}_t L^{6+}_x} \| u_5\|_{L^{\infty}_t L^2_x}
\end{align*}
leading to the same bound.

\noindent\textbf{Case 2:} $N_5 \ge N_4 \gtrsim N$.
The additional factor $(\frac{N_4}{N})^{\frac{1}{2}-}(\frac{N_5}{N})^{\frac{1}{2}-}$ 
is compensated by replacing the last factor in \eqref{54} by
\begin{align*}
&N_4^{-1/2} N_5^{-1/2}\|\langle D \rangle^{-2} D^{-1/2} u_4
 D^{-1/2} u_5\|_{L^{\infty}_t L^{3}_x} \\
& \lesssim N_4^{-1/2} N_5^{-1/2} 
\|D^{-1/2}u_4\|_{L^{\infty}_t L^3_x} \|D^{-1/2} u_5\|_{L^{\infty}_t L^3_x} \\
&\lesssim  N_4^{-1/2} N_5^{-1/2} \|u_4\|_{L^{\infty}_t L^2_x} 
\| u_5\|_{L^{\infty}_t L^2_x} \,.
\end{align*}

\noindent\textbf{Case 3:} $N_5 \ge N \ge N_4$.
Replace the last factor in \eqref{54} by
\begin{align*}
\|\langle D \rangle^{-2} D^{-1} u_4 D^{-1} u_5\|_{L^{\infty}_t L^{3+}_x} 
& \lesssim \|D^{-1}u_4\|_{L^{\infty}_t L^{6+}_x} \|D^{-1} u_5\|_{L^{\infty}_t L^6_x} \\
&\lesssim  N_4^{0+} \|u_4\|_{L^{\infty}_t L^2_x} \| u_5\|_{L^{\infty}_t L^2_x} \,.
\end{align*}
Thus \eqref{54''} is proven.

The next claim is
\begin{equation}
\label{54'''}
 \big|\int_0^{\delta} \langle (W*|Iu|^2)-I(W*|u|^2), I((W*|u|^2)u) \rangle dt\big| 
\lesssim \frac{\delta}{N^{2-}} \|\nabla Iu\|^5_{X^{0,\frac{1}{2}+}}  .
 \end{equation}
We again consider $D$ with
\begin{align*}
&M(\xi_1,\dots .,\xi_5)\\
& :=  \frac{|m(\xi_1+\xi_2)-m(\xi_1)m(\xi_2)|}{m(\xi_1)m(\xi_2) 
\langle \xi_1 + \xi_2 \rangle^2 }
\cdot \frac{m(\xi_3 +\xi_4 + \xi_5)}{m(\xi_3)m(\xi_4)m(\xi_5) \langle \xi_3 
+ \xi_4 \rangle^2} \cdot \prod_{i=1}^5  |\xi_i|^{-1} \,.
\end{align*}
We assume without loss of generality $N_2 \ge N_1$ and $N_3 \ge N_4$.

\noindent\textbf{Case 1:} $ N \gg N_3,N_4,N_5$ 
($\Longrightarrow N_1 \sim N_2 \gtrsim N$).
In this case we obtain
\begin{align*}
D & \lesssim (\frac{N_1}{N})^{1/2} (\frac{N_2}{N})^{1/2} 
  \frac{\delta}{N_1 N_2} \| \langle D \rangle^{-2}( u_1 u_2)\|_{L^{\infty}_t 
 L^{\frac{6}{5}}_x} \\
& \quad\times \| \langle D \rangle^{-2}(D^{-1} u_3 D^{-1} u_4)\|_{L^{\infty}_t
 L^{\infty-}_x}  
 \|D^{-1} u_5\|_{L^{\infty}_t L^{6+}_x} \\
& \lesssim \frac{\delta}{N_{\max}^{0+} N^{2-}} \| u_1 u_2\|_{L^{\infty}_t L^1_x}
  \|D^{-1} u_3 D^{-1} u_4\|_{L^{\infty}_t L^{3+}_x} \|D^{-1} u_5\|_{L^{\infty}_t
  L^{6+}_x} \\
&  \lesssim \frac{\delta(1 \wedge N_{\min}^{0+})}{N_{\max}^{0+} N^{2-}} 
\prod_{i=1}^5 \|u_i\|_{X^{0,\frac{1}{2}+}} \,.
\end{align*}

\noindent\textbf{Case 2:} $ N_3 \ge N \gg N_4,N_5$.\\
\textbf{a.} $N_3 \ge N \gg N_1$. 
We obtain
\begin{align*}
D & \lesssim \frac{\delta}{N_2 N_3} \| \langle D 
\rangle^{-2}(D^{-1}u_1 u_2)\|_{L^{\infty}_t L^2_x} \| 
\langle D \rangle^{-2}(u_3 D^{-1} u_4)\|_{L^{\infty}_t L^{3-}_x}  
 \|D^{-1} u_5\|_{L^{\infty}_t L^{6+}_x} \\
& \lesssim \frac{\delta}{ N_2 N_3} \|D^{-1} u_1 u_2\|_{L^{\infty}_t L^{\frac{3}{2}+}_x}  \|u_3 D^{-1} u_4\|_{L^{\infty}_t L^{\frac{3}{2}+}_x} \|D^{-1} u_5\|_{L^{\infty}_t L^{6+}_x} \\
&  \lesssim \frac{\delta(1 \wedge N_{\min}^{0+})}{N_{\max}^{0+} N^{2-}} \prod_{i=1}^5 \|u_i\|_{X^{0,\frac{1}{2}+}} \,.
\end{align*}
\textbf{b.} $N_2 \ge N_1 \gtrsim N$.
We obtain an additional factor 
$(\frac{N_1}{N})^{\frac{1}{2}-} (\frac{N_2}{N})^{\frac{1}{2}-}$, which is acceptable, 
when one estimates as in a.

\noindent \textbf{Case 3:} $ N_5 \ge N \gg N_3,N_4$ ($\Longrightarrow N_2 \gtrsim N$).
\\
\textbf{a.} $N_1 \le N$. We obtain
\begin{align*}
D & \lesssim \frac{\delta}{N_2 N_5} \| \langle D 
\rangle^{-2}(D^{-1}u_1 u_2)\|_{L^{\infty}_t L^2_x} \| 
\langle D \rangle^{-2}(D^{-1}u_3 D^{-1} u_4)\|_{L^{\infty}_t L^{\infty}_x} 
\|u_5\|_{L^{\infty}_t L^{2}_x} \\
& \lesssim \frac{\delta}{N_2 N_5} \|D^{-1} u_1 u_2\|_{L^{\infty}_t L^{\frac{3}{2}+}_x}
  \|D^{-1}u_3 D^{-1} u_4\|_{L^{\infty}_t L^{3+}_x} \| u_5\|_{L^{\infty}_t L^{2}_x} \\
&  \lesssim \frac{\delta(1 \wedge N_{\min}^{0+})}{N_{\max}^{0+} N^{2-}} 
\prod_{i=1}^5 \|u_i\|_{X^{0,\frac{1}{2}+}} \,.
\end{align*}
\textbf{b.} $N_1 \ge N$.
We obtain an additional factor 
$(\frac{N_1}{N})^{\frac{1}{2}-} (\frac{N_2}{N})^{\frac{1}{2}-}$,
 which can be compensated as in a.

\noindent\textbf{Case 4:} $N_3,N_4 \gtrsim N \gtrsim N_5$.\\
\textbf{a.} $N_2 \ge N \gg N_1$.
We obtain
\begin{align*}
D & \lesssim (\frac{N_3}{N})^{1/2} (\frac{N_4}{N})^{1/2} 
 \frac{\delta}{N_2 N_3 N_4} \| \langle D \rangle^{-2}(D^{-1}u_1 u_2)\|_{L^{\infty}_t 
 L^2_x} \| \langle D \rangle^{-2}(u_3 u_4)\|_{L^{\infty}_t L^{3-}_x}  \\
& \quad\times  \|D^{-1}u_5\|_{L^{\infty}_t L^{6+}_x} \\
& \lesssim  (\frac{N_3}{N})^{1/2} (\frac{N_4}{N})^{1/2} \frac{\delta}{N_2 N_3 N_4} 
 \|D^{-1} u_1 u_2\|_{L^{\infty}_t L^{\frac{3}{2}+}_x}  \|u_3 u_4\|_{L^{\infty}_t 
 L^1_x} \|D^{-1} u_5\|_{L^{\infty}_t L^{6+}_x} \\
&  \lesssim \frac{\delta(1 \wedge N_{\min}^{0+})}{N_{\max}^{0+} N^{3-}}
  \prod_{i=1}^5 \|u_i\|_{X^{0,\frac{1}{2}+}} \,.
\end{align*}
\textbf{b.} $N_1 \ge N$.
We obtain an additional factor 
$(\frac{N_1}{N})^{\frac{1}{2}-} (\frac{N_2}{N})^{\frac{1}{2}-}$,
 which can be compensated by replacing the term
$ \|D^{-1} u_1 u_2\|_{L^{\infty}_t L^{\frac{3}{2}+}_x} $ in a. by
$$ 
\|D^{-1} u_1 u_2\|_{L^{\infty}_t L^1_x} \lesssim \frac{1}{N_1} \|u_1\|_{L^{\infty}_t 
L^2_x}\|u_2\|_{L^{\infty}_t L^2_x} \,. 
$$

\noindent\textbf{Case 5:} $N_3,N_5 \ge N \ge N_4$.\\
\textbf{a.} $N_2 \ge N \ge N_1$.
We have
\begin{align*}
D & \lesssim (\frac{N_3}{N})^{1/2} (\frac{N_5}{N})^{1/2} \frac{\delta}{ N_3 N_5} 
\| \langle D \rangle^{-2}(D^{-1}u_1 D^{-1}u_2)\|_{L^{\infty}_t L^{3+}_x}\\
& \quad\times \| 
\langle D \rangle^{-2}(u_3 D^{-1} u_4)\|_{L^{\infty}_t L^{6-}_x}  
\|u_5\|_{L^{\infty}_t L^2_x} \\
& \lesssim \frac{\delta}{N_{\max}^{0+} N^{2-}}  \|D^{-1} u_1 D^{-1}u_2\|_{L^{\infty}_t 
 L^{3+}_x}  \|u_3 D^{-1}u_4\|_{L^{\infty}_t L^{\frac{3}{2}+}_x} \| u_5\|_{L^{\infty}_t
 L^2_x} \\
&  \lesssim \frac{\delta(1 \wedge N_{\min}^{0+})}{N_{\max}^{0+} N^{2-}} \prod_{i=1}^5 
\|u_i\|_{X^{0,\frac{1}{2}+}} \,.
\end{align*}
\textbf{b.} $N_2 \ge N_1 \ge N$.
We obtain an additional factor $(\frac{N_1}{N})^{\frac{1}{2}-} 
(\frac{N_2}{N})^{\frac{1}{2}-}$, which can be compensated by replacing the term
$ \|D^{-1} u_1 D^{-1}u_2\|_{L^{\infty}_t L^{3+}_x} $ in a. by
$$ 
\|D^{-1} u_1 D^{-1}u_2\|_{L^{\infty}_t L^{1+}_x} 
\lesssim \frac{1}{N_1^{1-} N_2^{1-}} \|u_1\|_{L^{\infty}_t L^2_x}\|u_2\|_{L^{\infty}_t 
L^2_x} \,. 
$$

\noindent\textbf{Case 6:} $N_3,N_4,N_5 \ge N$.\\
\textbf{a.} $N_2 \ge N \ge N_1$.
We have
\begin{align*}
D & \lesssim (\frac{N_3}{N})^{1/2} (\frac{N_4}{N})^{1/2}(\frac{N_5}{N})^{1/2} \frac{\delta}{ N_3 N_4 N_5} \| \langle D \rangle^{-2}(D^{-1}u_1 D^{-1}u_2)\|_{L^{\infty}_t L^{6}_x} \\
& \quad\times  \| \langle D \rangle^{-2}(u_3 u_4)\|_{L^{\infty}_t L^{3}_x}
 \|u_5\|_{L^{\infty}_t L^2_x} \\
& \lesssim \frac{\delta}{N_{\max}^{0+} N^{3-}}  \|D^{-1} u_1 D^{-1}u_2\|_{L^{\infty}_t L^{3+}_x}  \|u_3 u_4\|_{L^{\infty}_t L^1_x} \| u_5\|_{L^{\infty}_t L^2_x} \\
&  \lesssim \frac{\delta(1 \wedge N_{\min}^{0+})}{N_{\max}^{0+} N^{3-}} \prod_{i=1}^5 \|u_i\|_{X^{0,\frac{1}{2}+}} \,.
\end{align*}
\textbf{b.} $N_2 \ge N_1 \ge N$ ($\Rightarrow N_{\min} \gtrsim N$).
We obtain an additional factor 
$(\frac{N_1}{N})^{\frac{1}{2}-} (\frac{N_2}{N})^{\frac{1}{2}-} 
\lesssim (\frac{N_1}{N})^{1-}$ leading to
\begin{align*}
D & \lesssim \frac{\delta}{ N_{\max}^{0+} N^{4-}} \| \langle D 
 \rangle^{-2}(u_1 D^{-1}u_2)\|_{L^{\infty}_t L^{6}_x} \| \langle D 
\rangle^{-2}(u_3 u_4)\|_{L^{\infty}_t L^{3}_x}  
 \|u_5\|_{L^{\infty}_t L^2_x} \\
& \lesssim \frac{\delta}{N_{\max}^{0+} N^{4-}}  \| u_1 D^{-1}u_2\|_{L^{\infty}_t
  L^{3/2}_x}  \|u_3 u_4\|_{L^{\infty}_t L^1_x} \| u_5\|_{L^{\infty}_t L^2_x} \\
&  \lesssim \frac{\delta(1 \wedge N_{\min}^{0+})}{N_{\max}^{0+} N^{4-}} 
 \prod_{i=1}^5 \|u_i\|_{X^{0,\frac{1}{2}+}} \,.
\end{align*}
which completes the proof of \eqref{54'''}.

Next we want to prove the following estimate:
\begin{equation}
\label{54''''}
 \big|\int_0^{\delta} \langle (W* \operatorname{Re} Iu)Iu-I((W* \operatorname{Re} u)u), 
I((W*|u|^2)u) \rangle dt\big| 
\lesssim \frac{\delta}{N^{2-}}  \|\nabla Iu\|^5_{X^{0,\frac{1}{2}+}}  .
 \end{equation}
We again consider $D$ with
\begin{align*}
&M(\xi_1,\dots .,\xi_5) \\
&:=  \frac{|m(\xi_1+\xi_2)-m(\xi_1)m(\xi_2)|}{m(\xi_1)m(\xi_2) 
\langle \xi_1 \rangle^2 }\cdot \frac{m(\xi_3 +\xi_4
 + \xi_5)}{m(\xi_3)m(\xi_4)m(\xi_5) \langle \xi_3 + \xi_4 \rangle^2}
\cdot \prod_{i=1}^5  |\xi_i|^{-1} \,,
\end{align*}
and treat only the more difficult case $N_2 \ge N_1$, and assume without 
loss of generality $N_3 \ge N_4$. We consider the same cases as for \eqref{54'''}. 

In case 1 we replace 
$\|\langle D \rangle^{-2}(u_1 u_2)\|_{L^{\infty}_t L^{\frac{6}{5}}_x}$ by
 $$
\|\langle D \rangle^{-2}u_1 u_2\|_{L^{\infty}_t L^{\frac{6}{5}}_x} 
\lesssim \|u_1\|_{L^{\infty}_t L^2_x} \|u_2\|_{L^{\infty}_t L^2_x} \,.
$$
In the cases 2, 3 and 4a we replace  
$\|\langle D \rangle^{-2}(D^{-1}u_1 u_2)\|_{L^{\infty}_t L^2_x}$ by
 $$
\|\langle D \rangle^{-2} D^{-1} u_1 u_2\|_{L^{\infty}_t L^2_x}
 \lesssim \|D^{-1} u_1\|_{L^{\infty}_t L^{6+}_x} \|u_2\|_{L^{\infty}_t L^2_x} \,.
$$
In case 4b. estimate
 $$
\|\langle D \rangle^{-2} D^{-1} u_1 u_2\|_{L^{\infty}_t L^2_x}
 \lesssim \| u_1\|_{L^{\infty}_t L^2_x} \|u_2\|_{L^{\infty}_t L^2_x} \,.
$$
Case 5a is essentially unchanged, whereas in Case 5b replace 
$$
\|\langle D \rangle^{-2}(D^{-1}u_1 D^{-1}u_2)\|_{L^{\infty}_t L^{3+}_x} 
\|\langle D \rangle^{-2}(u_3 D^{-1}u_4)\|_{L^{\infty}_t L^{6-}_x}
$$
 by
\begin{align*}
&\|\langle D \rangle^{-2}D^{-1}u_1 D^{-1} u_2\|_{L^{\infty}_t L^{2+}_x} 
\|\langle D \rangle^{-2}(u_3 D^{-1}u_4)\|_{L^{\infty}_t L^{\infty -}_x} \\
& \lesssim \|D^{-1}u_1\|_{L^{\infty}_t L^2_x} \|D^{-1}u_2\|_{L^{\infty}_t
 L^{2+}_x} \|u_3 D^{-1}u_4\|_{L^{\infty}_t L^{\frac{3}{2}+}_x} \\
&\lesssim  \frac{1}{N_1 N_2^{1-}} \|u_1\|_{L^{\infty}_t L^2_x} 
\|u_2\|_{L^{\infty}_t L^2_x}
 \|u_3\|_{L^{\infty}_t L^2_x} \|D^{-1}u_4\|_{L^{\infty}_t L^{6+}_x} \,.
 \end{align*}
In case 6a estimate
 \begin{align*}
\|\langle D \rangle^{-2} D^{-1} u_1 D^{-1} u_2\|_{L^{\infty}_t L^6_x} 
&\lesssim \|\langle D \rangle^{-2} D^{-1} u_1\|_{L^{\infty}_t L^{\infty}_x} \|D^{-1} u_2\|_{L^{\infty}_t L^6_x} \\
& \lesssim \|D^{-1} u_1\|_{L^{\infty}_t L^{6+}_x} \| u_2\|_{L^{\infty}_t L^2_x}\,.
\end{align*}
Similarly case 6b can be handled, so that \eqref{54''''} is complete.

The forth and third order terms in $|\langle F(Iu)-IF(u),IF(u)\rangle|$ 
turn out to be less critical. We omit any detailed calculations here and 
just refer to the recent paper of the author \cite{Pe}, where the following 
estimates were given even under the weaker assumption $|\widehat{W}(\xi)| \lesssim 1$ 
(compared to the property $|\widehat{W}(\xi)| \lesssim \langle \xi \rangle^{-2}$
 which we have in the present study). We have to remark that the assumption
 $s \ge \frac{3}{4}$ in that paper is not really necessary for these forth and 
third order terms, but could be replaced by $s > 1/2$, because the 
factors $(\frac{N_i}{N})^{\frac{1}{4}}$ can everywhere be replaced by 
$(\frac{N_i}{N})^{\frac{1}{2}-}$ without significance for the results. 
We have (\cite{Pe}, section 4.6, 4.7 and 4.8):
\begin{gather*}
\int_0^{\delta} |\langle I(u^3)-(Iu)^3,Iu \rangle| dt 
 \lesssim N^{-3+} \|\nabla Iu\|^4_{X^{0,\frac{1}{2}+}[0,\delta]} \,,\\
\int_0^{\delta} |\langle I(u^2)-(Iu)^2,(Iu)^2 \rangle| dt  
\lesssim N^{-3+} \|\nabla Iu\|^4_{X^{0,\frac{1}{2}+}[0,\delta]} \,,\\
\int_0^{\delta} |\langle I(u^2)-(Iu)^2,Iu \rangle| dt  
\lesssim N^{-\frac{5}{2}+} \delta^{1/2} \|\nabla Iu\|^3_{X^{0,\frac{1}{2}+}[0,\delta]} \,.
\end{gather*}
Summarizing all our estimates in this chapter we arrive at \eqref{29'}.


\begin{thebibliography}{00}

\bibitem{ABJ} A. Aftalion, X. Blanc and R. L. Jerrard: 
\emph{Nonclassical rotational inertia of a supersolid}. Phys. 
Review Letters 99 (2007), 135301-4.

\bibitem{BS} F. Bethuel and J. C. Saut: 
\emph{Travelling waves for the
Gross-Pitaevskii equation I}. Ann. I. H. Poincar\'e Phys. Th\'eor. 70 (1999),
147-238.

\bibitem{CH} T. Cazenave and A. Haraux: 
\emph{An introduction to semilinear evolution equations}. 
Oxford science publications 1998.

\bibitem{CKSTT} J. Colliander, M. Keel, G. Staffilani, H. Takaoka and
T. Tao:
\emph{Almost conservation laws and global rough solutions to a nonlinear
Schr\"odinger equation}. Math. Res. Letters 9 (2002), 659-682.

\bibitem{C} C. Coste:
 \emph{Nonlinear Schr\"odinger equation and superfluid hydrodynamics}. 
Eur. Phys. J. B Condens. Matter Phys. (1998), 245-253.

\bibitem{Ga} C. Gallo:
 \emph{The Cauchy problem for defocusing nonlinear
Schr\"odinger equations with non-vanishing initial data at infinity}. Comm.
Part. Diff. Equa. 33 (2008), 729-771.

\bibitem{Ge} P. G\'erard:
\emph{The Cauchy problem for the Gross-Pitaevskii equation}.
 Ann. I. H. Poincar\'e Anal. Non-lin\'eaire 23 (2006), 765-779.

\bibitem{GTV} J. Ginibre, Y. Tsutsumi and G. Velo: \emph{On the Cauchy
problem for the Zakharov system}. J. Funct. Analysis 151 (1997), 384-436

\bibitem{Gr} E.P. Gross: \emph{Hydrodynamics of a Superfluid Condensate}.
J. Math. Phys. 4 (1963), 195-207.

\bibitem{G} A. Gr\"unrock:
 \emph{New applications of the Fourier restriction
norm method to wellposedness problems for nonlinear evolution equations}.
Dissertation Univ. Wuppertal 2002, \\
{\tt http://elpub.bib.uni-wuppertal.de/servlets/DocumentServlet?id=254}.

\bibitem{JPR} Ch. Josserand, Y. Pomeau and S. Rica:
 \emph{Coexistence of ordinary elasticity and superfluidity in a model of 
a defect-free supersolid}. Phys. Review Letters 98 (2007), 195301-4.

\bibitem{JPRo} C. A. Jones, S. J. Putterman, P. H. Roberts: 
\emph{Motion in a Bose condensate V. Stability of solitary wave solutions 
of non-linear Schr\"odinger equations in two and three dimensions}. 
J. Phys. A, Math. Gen. 19 (1986), 2991-3011.

\bibitem{JR} C. A. Jones and P. H. Roberts: 
\emph{Motions in a Bose condensate IV. Axisymmetric solitary waves}. 
J. Phys. A, Math. Gen. 15 (1982), 2599-2619.

\bibitem{KT} M. Keel and T. Tao: 
\emph{Endpoint Strichartz estimates}.
Amer. J. Math. 120 (1998), 955-98.

\bibitem{KOPV} R. Killip, T. Oh, O. Pocovnicu and M. Visan:
 \emph{Global well-posedness of the Gross-Pitaevskii and cubic-quintic nonlinear 
Schrödinger equations with non-vanishing boundary conditions}. 
 Preprint arXiv:1112.1354.

\bibitem{KL} Y.S. Kivshar and B. Luther-Davies:
 \emph{Dark optical solitons: physics and applications}. 
Phys. Rep. 298 (1998), 81-197.

\bibitem{L} A. de Laire: 
\emph{Global well-posedness for a nonlocal Gross-Pitaevskii equation with 
non-zero condition at infinity}. Comm in PDE 35 (2010), 2021-2058.

\bibitem{Pe} H. Pecher: 
\emph{Unconditional global well-posedness for the 3D Gross-Pitaevskii equation
 for data without finite energy}. arXiv:1201.3777.

\bibitem{P} L.P. Pitaevskii: \emph{Vortex lines in an imperfect Bose gas}.
Soviet Physics JETP 13 (1961), 451-454.

\bibitem{PR} Y. Pomeau and S. Rica:
 \emph{Model of superflow with rotons}. Phys. Review Letters 71 (1993), 247-250.

\bibitem{SK} V. S. Shchesnovich and R. A. Kraenkel: 
\emph{Vortices in lonlocal Gross-Pitaevskii equation}. 
J. Phys. A 37 (2004), 6633-6651.

\end{thebibliography}
\end{document}